\documentclass[12pt,reqno]{amsart}
\usepackage {amssymb}
\usepackage {amsmath}
\usepackage {bbm}
\usepackage{amsthm}
\usepackage{mathtools}
\usepackage{graphicx}
\usepackage {amscd}
\usepackage[alphabetic]{amsrefs}
\usepackage[colorlinks, citecolor=blue]{hyperref}
\usepackage{enumerate}
\usepackage{tikz}
\usepackage{tikz-cd}
\usepackage{verbatim,color,geometry}
\usepackage[all]{xy}
\newtheorem{prop}{Proposition}[section]
\newtheorem{thm}[prop]{Theorem}
\newtheorem{cor}[prop]{Corollary}

\newtheorem{lem}[prop]{Lemma}

\theoremstyle{definition}

\newtheorem{defn}[prop]{Definition}
\newtheorem{defn-lem}[prop]{Definition-Lemma}

\newtheorem{rem}[prop]{\it Remark}

\newtheorem{claim}[prop]{Claim}

\newtheorem*{claim*}{Claim}
\newtheorem{say}[prop]{}

\newcommand{\bR}{\mathbb{R}}
\newcommand{\bA}{\mathbb{A}}
\newcommand{\bQ}{\mathbb{Q}}
\newcommand{\bZ}{\mathbb{Z}}
\newcommand{\bN}{\mathbb{N}}

\newcommand{\hvol}{\widehat{\rm vol}}

\newcommand{\cD}{\mathcal{D}}

\newcommand{\cO}{\mathcal{O}}

\newcommand{\fa}{\mathfrak{a}}
\newcommand{\fb}{\mathfrak{b}}

\newcommand{\fm}{\mathfrak{m}}

\newcommand{\Spec}{\mathrm{Spec}~}

\newcommand{\mult}{\mathrm{mult}}
\newcommand{\lct}{\mathrm{lct}}

\newcommand{\vol}{\mathrm{vol}}
\newcommand{\ord}{\mathrm{ord}}

\newcommand{\Val}{\mathrm{Val}}

\newcommand{\QM}{\mathrm{QM}}
\newcommand{\Cosupp}{\mathrm{Cosupp}}

\numberwithin{equation}{section}

\begin{document}

\title{A Minimizing valuation is quasi-monomial}

\author{Chenyang Xu}
\address{Department of Mathematics, MIT, Cambridge, MA, USA, 02139.}
\email{cyxu@math.mit.edu}

\address{BICMR, Peking University, Beijing, China 100871.}
\email{cyxu@math.pku.edu}

\date{}

\maketitle
{\let\thefootnote\relax\footnotetext{CX is partially supported by a Chern Professorship of the MSRI (NSF No. DMS-1440140) and by the
NSF (No. 1901849). }
\marginpar{}
}

\begin{abstract}{
We prove a version of Jonsson-Musta\c{t}\v{a}'s Conjecture, which says for any graded sequence of ideals, there exists a quasi-monomial valuation computing its log canonical threshold. As a corollary, we confirm Chi Li's conjecture that a minimizer of the normalized volume function is always quasi-monomial. 

Applying our techniques to a family of klt singularities, we show that the volume of klt singularities is a constructible function. As a corollary, we prove that in a family of klt log Fano pairs, the K-semistable ones form a Zariski open set. Together with \cite{Jiang-bounded}, we conclude that all K-semistable klt Fano varieties with a fixed dimension and volume are parametrized by an Artin stack of finite type, which then admits a separated good moduli space by \cite{BX-uniqueness, ABHX-reductivity}, whose geometric points parametrize K-polystable klt Fano varieties.}
\end{abstract}

\setcounter{tocdepth}{1}
\tableofcontents
\section{Introduction}

Throughout this paper, we work over an algebraically closed field $k$ of characteristic 0. In this note, we use recent developments in birational geometry, especially results from the minimal model program, to study invariants of singularities which are of an asymptotic nature. We aim to get some uniform results which can not be obtained by previous methods. 

\subsection{The valuation computing the log canonical threshold}
Our first theorem is to prove that for any graded sequence of ideals, there always exists a quasi-monomial valuation which computes the log canonical threshold.
\begin{thm}[{\cite[Conjecture B]{JM-qmvaluations}}]\label{t-JMconjecture}
If $ (X,\Delta)$ is a Kawamata log terminal (klt) pair, and $\fa_{\bullet}:=\{\fa_k\}_{k\in \bN}$ is a graded sequence of ideals such that ${\rm lct}(\fa_{\bullet})<\infty$. 
Then there exists a quasi-monomial valuation $v$ which  calculates the log canonical threshold of $\fa_{\bullet}$, i.e.,
$$\frac{A_{X,\Delta}(v)}{v(\fa_{\bullet})}=\inf_w \frac{A_{X,\Delta}(w)}{w(\fa_{\bullet})},$$
where $w$ runs through all valuations whose center is on $X$. 
\end{thm}

This confirms the weak version of  \cite[Conjecture B]{JM-qmvaluations}. However, our techniques does not directly give the strong version, which predicts {\it any} valuation $w$ computing the log canonical threshold of $\fa_{\bullet}$ is quasi-monomial. More precisely, for any such $w$, our approach produces a quasi-monomial valuation $v$ with $A_{X,\Delta}(v)=A_{X,\Delta}(w)$, $v\ge w$ and $v$ also computes the log canonical threshold of $\fa_{\bullet}$.

One consequence of the above theorem is the following statement. 
\begin{thm}[{\cite[Conjecture 7.1.3]{Li-minimizer}}]\label{t-minimizing}
Let $x\in (X,\Delta)$ be a klt singularity. Any minimizer $v^{\rm m}$ of the normalized volume function 
$$\hvol_{(X,\Delta),x}\colon \Val_{X,x}\to \bR_{>0}\bigcup\{+\infty\}$$ is quasi-monomial. 
\end{thm}
This is one piece of a circle of conjectures about the minimizer of the normalized volume function $\hvol_{(X,\Delta),x}$, which are all together packed into the Stable Degeneration Conjecture (see \cite[Conj. 7.1]{Li-minimizer}, \cite[Conj. 1.2]{LX-SDCII}), and predict some deep information about an arbitrary klt singularity. We note that the Stable Degeneration Conjecture has been intensively studied (see e.g.  \cite{Li-K=M, Blu-Existence, LL-K=M, LX-SDCI, LX-SDCII}). Combining Theorem \ref{t-minimizing} with the previously known results, the main remaining part  is to show that for the quasi-monomial minimizer $v$, its associated graded ring is finitely generated. While this is known when the rational rank is one (\cite{LX-SDCI, Blu-Existence}), this is still open when the rational rank of $v$ is larger than one, except when $\dim(X)=2$ (see \cite[Prop. 1.4]{Cutkosky-twodimensional}).

Another application of Theorem \ref{t-JMconjecture} is that it finishes the algebraic approach, originated from \cite[Theorem D]{JM-openness}, of solving the Demailly-Koll\'ar's Openness Conjecture (see \cite{DK-lct}). Recall the Openness Conjecture says that  for any germ of a pluri-subharmonic function $\phi$ at a point $x$ on a complex manifold, denoted by $c_x(\phi)$ the complex singularity exponent of $\phi$ at $x$ (see \cite[Def. 0.1]{DK-lct}),  we have
$$\mbox{ if $c_x(\phi) <\infty$, then the function ${\rm exp}(-2c_x(\phi)\phi)$ is not locally integrable at $x$}.$$
 We note that the Openness Conjecture has been proved in \cite{Berndtsson-openness, GZ-openness} by completely different methods with more analytic nature. On the other hand, it seems that our approach can {\it not} yield the stronger version \cite[Conjecture C$^{\prime\prime}$]{JM-openness}, which would imply the Strong Openness Conjecture. Nevertheless, the latter is also proved in \cite{GZ-openness} (see also \cite{Hiep-openness, Lempert-openness}).

\bigskip

The proof of Theorem \ref{t-JMconjecture} depends on a combination of two sets of recently established techniques: The first one is approximating a valuation computing the log canonical threshold by a sequence of valuations which can be better understood using birational geometry. This idea is developed in \cite{LX-SDCI}. In particular, the proof of \cite[Theorem 1.3]{LX-SDCI} essentially implies that in Theorem \ref{t-JMconjecture}, we can find a valuation $v$ computing the log canonical threshold which can be always approximated by a sequence of rescalings of Koll\'ar components $S_i$ (see Definition \ref{d-kollarcomp}). Roughly speaking, Koll\'ar components are divisorial valuations over $x\in (X,\Delta)$ which admits a log Fano structure. 

The second main ingredient is the boundedness of complements which was recently established in \cite{Bir-BABI}. This difficult result together with an estimate established in \cite{Li-minimizer},  imply that all ${S_i}$ can be obtained as log canonical places of {\it a bounded family of $\mathbb{Q}$-Cartier divisors} on $(X,\Delta)$. From this boundedness, we then could conclude that the limit  is quasi-monomial. 

\subsection{The volume function of klt singularities is constructible}

Applying our techniques to a family of klt singularities also leads to a proof of the following result.
\begin{thm}\label{t-volumeconstruct}
For a $\bQ$-Gorenstein family of  klt singularities $(B\subset (X,\Delta))\to B$ over a smooth base, the volume function 
$$\hvol_B\colon B\to \bR_{>0} \qquad (s\to B) \to \hvol({s}, X_{{s}},\Delta_{{s}}),$$ 
which sends each geometric point $s$ to the volume of singularity over $s$, 
is constructible in Zariski topology.
\end{thm}

In \cite{BL-volumelower}, it is shown that $\hvol_B$ is lower semi-continuous. 
Combining with the cone construction, we obtain the following theorem. 

\begin{thm}\label{t-openness}
For a $\mathbb{Q}$-Gorenstein family of log Fano pairs $ (X,\Delta)\to B$ over a smooth base $B$, the locus $B^{\circ}\subset B$ which parametrises K-semistable geometric fibers form an open set. 
\end{thm}

The openness of uniform K-stability in a $\mathbb{Q}$-Gorenstein family of log Fano pairs was previously proved in \cite{BL-openness}.
Together with \cite{Jiang-bounded, BX-uniqueness, ABHX-reductivity}, we conclude the following theorem. 

\begin{thm}\label{t-goodmoduli}
Fix $n$ and $V$, the functor $\mathfrak{X}^{\rm kss}_{n,V}$ of  families of K-semistable $\bQ$-Fano varieties of dimension $n$ and volume $V$ is an Artin stack of finite type. 

Moreover, it admits a good moduli space ${X}^{\rm kps}_{n,V}$, whose geometric points parametrise K-polystable $\bQ$-Fano varieties.
\end{thm}

We expect that ${X}^{\rm kps}_{n,V}$ is proper or even projective. 

\begin{rem}In the simultaneous work \cite{BLX-openness}, in a more global setting, i.e., for log Fano pairs $(X,\Delta)$, a similar strategy is applied to study the stability thresholds $\delta(X,\Delta)$. As a result, Theorem \ref{t-openness} as well as the consequence Theorem \ref{t-goodmoduli} are also proved there.
\end{rem}

\noindent {\bf Acknowledgement:} The author is grateful to Harold Blum, Mattias Jonsson, Chi Li,
Yuchen Liu,  Davesh Maulik, Mircea Musta\c{t}\v{a}, and Ziquan Zhuang for helpful
conversations. 
He also thanks Harold Blum for reading the first draft of the paper and giving valuable comments, as well as the anonymous referee for many suggestions to significantly improve the exposition. 
A large part of the
work on this paper was completed when CX visited MSRI. He thanks the institute for the wonderful environment. 
\section{Preliminaries}
\noindent {\bf Notation and Conventions:} We use the standard notation as in \cite{Laz-positivityII}, \cite{KollarMori98} and \cite{Kollar13}. In particular, for $X$ a normal variety, with a $\bQ$-divisor $\Delta$ on $X$, we define $(X,\Delta)$ to be {\it klt, lc} as in \cite[Definition 2.34]{KollarMori98}. For a log canonical pair $(X,\Delta)$, and a divisor $E$ over $(X,\Delta)$ its log discrepancy $A_{X,\Delta}(E)=1+a(E,X,\Delta)$ where $a(E,X,\Delta)$ is its discrepancy (see \cite[Definition 2.25]{KollarMori98}). We say a divisor $E$ over $X$ is a {\it log canonical place} if the log discrepancy $A_{X,\Delta}(E)=0$. We call $x\in X=\Spec R$ {\it a germ} if  $R$ is a local ring essentially of finite type over $k$ and $x$ is its closed point. When $x\in X$ is a germ of normal point, we can define $(X,\Delta)$ to be klt (resp. lc) for a $\bQ$-divisor in the obvious way. By abuse of notation, in this case, we say $x\in (X,\Delta)$ is a klt (resp. lc) singularity.

A projective pair $(X,\Delta)$ is called a {\it log Fano pair}, if $(X,\Delta)$ is klt and $-K_X-\Delta$ is ample. 

For a morphism $X\to B$ and a point $s\in B$, we will use $X_s$ to mean its fiber. 

Given a ring $R$, an $\bN$-graded sequence $\fa_{\bullet}=\{\fa_k\}_{k\in \bN}$  of ideals is a set of ideals $\fa_{k}\subset R$ $(k\in \bN)$ satisfying that $\fa_{k}\cdot \fa_{k'}\subset \fa_{k+k'}$. 
We will sometimes also include $\fa_0=R$ in a graded sequence of ideals. 

For two divisors If $D=\sum_i{d_i}D_i$ and $D'=\sum_id_i'D_i$,  we define $D\wedge D'=\sum_i\min\{d_i,d'_i\}D_i$.


\subsection{The space of valuations}
\subsubsection{Valuations}
Let $X$ be a reduced, irreducible (separated) variety defined over $k$. A \emph{real valuation} of its function field
$K(X)$ is a non-constant map $v\colon K(X)^{\times}\to \bR$, satisfying:
\begin{itemize}
 \item $v(fg)=v(f)+v(g)$;
 \item $v(f+g)\geq \min\{v(f),v(g)\}$;
 \item $v(k^*)=0$.
\end{itemize}
We set $v(0)=+\infty$. A valuation $v$ gives rise
to a valuation ring 
$$\cO_v:=\{f\in K(X)\mid v(f)\geq 0\}.$$
We say a real valuation $v$ is \emph{centered at} a scheme-theoretic
point $x=c_X(v)\in X$ if we have a local inclusion 
$\cO_{x,X}\hookrightarrow\cO_v$ of local rings.
Notice that the center of a valuation, if exists,
is unique since $X$ is separated. Denote by $\Val_X$ 
the set of real valuations of $K(X)$ that admits a center
on $X$. For a closed point $x\in X$, we denote by $\Val_{X,x}$ the set
of real valuations of $K(X)$ centered at $x\in X$. A valuation $v\in \Val_{X}$ is centered at $x\in X$ if $v(f)>0$ for any $f\in \fm_x$. 

For each valuation $v\in \Val_{X,x}$ and any positive integer $k$, we define the valuation ideal 
$$\fa^v_k:=\{f\in\cO_{x,X}\mid v(f)\geq k\}.$$ Then it is clear that $\fa^v_k$ is an $\fm_x$-primary ideal for each $k$ and $x=c_{X}(v)$.

\medskip

Given a valuation $v\in \Val_X$ and a nonzero ideal $\fa\subset\cO_X$, we may evaluate $\fa$ along $v$ by setting 
$$v(\fa) := \min\{v(f)\mid f \in \fa\cdot\cO_{c_X(v),X} \}.$$
It follows from the above definition that if $\fa\subset \fb \subset \cO_X$ are nonzero ideals, then $v(\fa) \geq v(\fb)$.
Additionally, $v(\fa)> 0$ if and only if $c_X(v) \in \Cosupp(\fa)$.
We endow $\Val_X$  with the weakest topology such that,
for every ideal $\fa$ on $X$, the map $\Val_X\to \bR\cup\{+\infty\}$ defined by $v\mapsto v(\fa)$ is continuous.
The subset $\Val_{X,x}\subset \Val_X$ is endowed with
the subspace topology. In some literatures, the space  $\Val_{X,x}$ is called the {\it non-archimedean link} of $x\in X$. For more background on valuation spaces, see \cite[Section 4]{JM-qmvaluations}.

We say two valuations $v\ge w$ if $v(\fa)\ge w(\fa)$ for any $\fa\subset \mathcal{O}_X$ (see \cite[Definition 4.3]{JM-qmvaluations}).
Let $D$ be a Cartier divisor on $X$ and $v$ a valuation, we can also define $v(D)$ to be $v(f)$ where $f$ is a defining equation of $D$ at $c_X(v)$. And similarly if $D$ is $\bQ$-Cartier, we can define $v(D)=\frac{1}{m}v(mD)$ for a sufficiently divisible positive integer $m$. For $\fa_{\bullet}=\{\fa_k\}_{k\in \bN}$, we define 
$$v(\fa_{\bullet})=\inf_{k}\frac{v(\fa_k)}{k}=\lim_{k\to \infty, \fa_k\neq 0}\frac{v(\fa_k)}{k}.$$


\medskip

Let $Y\xrightarrow[]{\mu} X$ be a proper birational morphism with $Y$  a normal variety. 
For a prime divisor $E$ on $Y$, we define a valuation $\ord_E\in \Val_X$ that sends each rational function in $K(X)^{\times}=K(Y)^{\times}$ to its order of vanishing along $E$. Note that 
the center $c_X(\ord_E)$ is the generic point of $\mu(E)$.
We say that $v\in \Val_X$ is a \emph{divisorial valuation} if there exists $E$ as above and $\lambda\in\bR_{>0}$ such that $v=\lambda\cdot \ord_E$.

\medskip

Next, we will introduce another important class of valuations which are called {\it quasi-monomial valuations}.
Let $\mu : Y \to X$ be a proper 
birational morphism and $\eta\in Y$ a point 
such that $Y$ is regular at $\eta$. Given a system of
parameters $y_1,\cdots, y_r \in \cO_{Y,\eta}$ at $\eta$
and $\alpha = (\alpha_1,\cdots,\alpha_r) \in \bR_{\geq 0}^r\setminus\{0\}$,
we define a valuation $v_\alpha$ as follows. For $f\in \cO_{Y,\eta}$
we can write it as $f =\sum_{\beta\in\bZ_{\geq 0}^r}c_\beta y^\beta$, with $c_\beta\in\widehat{\cO_{Y,\eta}}$  either zero or unit. We set
\begin{eqnarray}\label{e-quasimonomial}
v_\alpha(f) = \min\{\langle\alpha,\beta\rangle\mid c_\beta\neq 0\}.
\end{eqnarray}
A \emph{quasi-monomial valuation} is a valuation that can be written in the above form.

Let $(Y, E =\sum_{k=1}^N E_k)$ be a log smooth model of $X$, i.e. $Y$ is smooth, $E$ is snc  and  $\mu : Y \to X$ is an isomorphism outside  the support of $E$. We denote by $\QM_{\eta}(Y,E)$ the set of all
quasi-monomial valuations $v$ that can be described at the point $\eta\in Y$ with respect to
coordinates $(y_1,\cdots, y_r)$ such that each $y_i$ defines at $\eta$ an irreducible component of $E$ 
(hence $\eta$ is the generic point of a connected component of the intersection of some of the divisors $E_i$).
We put $\QM(Y,E):=\bigcup_{\eta}\QM_\eta(Y,E)\subset \Val_{X,x}$ where $\eta$ runs over
generic points of all irreducible components of intersections of 
some of the divisors $E_i$. Such a subspace $\QM(Y,E)$ can be naturally identified a cone over the dual complex $\cD(E)$ (see Definition \ref{d-dualcomplex}).

\bigskip

Given a valuation $v\in \Val_{X,x}$, its \emph{rational rank} ${\rm rat.rk}(v)$
is the rank of its value group. The \emph{transcendental degree}
${\rm trans}.\deg(v)$ of $v$ is the transcendental degree of the field extension
$k\hookrightarrow \cO_v/\fm_v$. 
Let $K$ be a field with transcendental degree $n$ over $k$, $k\subset K_0 \subset  K$ an intermediate field extension, $v$ a valuation on $K$
and $v_0$ its restriction to $K_0$. Then the Zariski-Abhyankar  inequality states that
\begin{eqnarray}\label{e-zainequ}
{\rm tr. deg}(v) + {\rm rat. rk}(v) \le {\rm tr. deg}(v_0) + {\rm rat. rk}(v_0) + {\rm tr. deg}(K/K_0). 
\end{eqnarray}
Taking $K_0=k$, we have
\[
 {\rm trans}.\deg(v)+{\rm rat.rk}(v)\leq n,
\]
and a valuation satisfying the equality is called an \emph{Abhyankar valuation}.
By \cite[Proposition 2.8]{ELS-localvolume}, we know that a valuation $v\in\Val_{X}$ is Abhyankar
if and only if it is quasi-monomial.

\bigskip

\subsubsection{Log discrepancy}

In the next, we give the definition of the log discrepancy $A_{X,\Delta}(v)$ (see Definition \ref{d-logd}).


\begin{defn}[Log discrepancy]\label{d-logd}
 Let $(X,\Delta)$ be a log canonical pair. We define the  \emph{(non-negative) log discrepancy
 function of valuations} $A_{X,\Delta}:\Val_X\to (0,+\infty]$
 in successive generality.
 \begin{enumerate}
  \item Let $\mu:Y\to X$ be a proper birational morphism from
  a normal variety $Y$. Let $E$ be a prime divisor
  on $Y$. Then we define $A_{X,\Delta}(\ord_E)=A_{X,\Delta}(E)$, i.e.
  \[
   A_{X,\Delta}(\ord_E):=1+\ord_E(K_Y-\mu^*(K_X+\Delta)).
  \]
  \item Let $(Y,E=\sum_{k=1}^N E_k)\to X$ be a log smooth model of $(X,\Delta)$, i.e. $Y$ is smooth and $E={\rm Supp}({\mu_*^{-1}}(\Delta)+{\rm Ex}(\mu))$ is snc. 
  Let $\eta$ be the generic point of a connected component of 
  $E_{i_1}\cap E_{i_2}\cap\cdots\cap E_{i_r}$ of codimension $r$. Let $(y_1,\cdots,y_r)$
  be a system of parameters of $\cO_{Y,\eta}$ at $\eta$ such that
  $E_{i_j}=(y_j=0)$. Then
  for any $\alpha=(\alpha_1,\cdots,\alpha_r)\in\bR_{\geq 0}^r\setminus\{0\}$, we define $A_{X,\Delta}(v_{\alpha})$ as
  \begin{eqnarray}\label{e-logdis}
   A_{X,\Delta}(v_\alpha):=\sum_{j=1}^r \alpha_j A_{X,\Delta}(\ord_{E_{i_j}}).
  \end{eqnarray}
  \item In \cite{JM-qmvaluations}, it was shown that there exists a retraction map 
  $$\rho_{Y}: \Val_X \to \QM(Y, E)$$ for any log smooth model $(Y, E)$ over $X$, such that it induces a
homeomorphism $\Val_X \to\varprojlim_{(Y,E)}\QM(Y, E)$. For any real valuation $v \in \Val_X$, we define
\begin{eqnarray}\label{e-logdis}
A_{X,\Delta}(v):=\sup_{(Y,E)} A_{X,\Delta}(\rho(v)),
\end{eqnarray}
where $(Y, E)$ ranges over all log smooth models over $(X,\Delta)$. For details, see \cite{JM-qmvaluations} and
\cite[Theorem 3.1]{BdFFU-valuation}. It is possible that $A_{X,\Delta}(v) = +\infty$ for some $v\in \Val_X$, see e.g. \cite[Remark 5.12]{JM-qmvaluations}.
\item For a lc pair $(X,\Delta)$ with an ideal $\fa\neq 0$ on $X$, for any $c\in \mathbb{Q}_{>0}$, we define 
$$A_{X,\Delta+c\cdot \fa}(v)=A_{X,\Delta}(v)-c\cdot v(\fa).$$	 
 \end{enumerate}
\end{defn}

In the above definition, if $(X,\Delta)$ is klt, then $A_{X,\Delta}$ is strictly positive on $\Val_{X}$. For a klt singularity $x\in (X,\Delta)$, we denote by $\Val^{=1}_{X,x}\subset  \Val_{X,x}$ the subspace consisting of all valuations  with $ A_{X,\Delta}(v)=1$.


\begin{defn}[Dual Complex]\label{d-dualcomplex}
For a simple normal crossing pair $(Y,E)$, we can form the dual complex $\cD(E)$ (see \cite[Definition 8]{dFKX-dualcomplex}). So every $E_i$ corresponds to a vertex $v_i$ and any component $Z$ of the intersection of $r$-components $E_{i_j}$ $(j=1,...r)$ corresponds to an $(r-1)$-cell 
$$W_Z:=\{x=(a_1,...,a_r)\in \mathbb{R}_{\ge 0}^{r}\ |\ \sum^r_{i=1}a_i=1\}$$ glued on $v_{j_1},..., v_{j_r}$. 

As a special case, for a log resolution $Y\to (X,\Delta)$ with the a set of exceptional divisors $E=\sum^r_{i=1}E_{i}$ over $x$, if $(X,\Delta)$ is klt, then there is a natural  embedding of 
$$i_{X,\Delta}\colon \cD(E) \to \Val^{=1}_{X,x}$$ by sending $v_i\to \frac{1}{A_{X,\Delta}(E_i)}\ord_{E_i}$ and the point on $W_Z$  with coordinates $(a_1,\cdots,a_r)$
$\big(\sum^r_{i=1}{a_i}=1\big)$  to the quasi-monomial valuation $v_{\alpha}\in \Val_{X,x}$ defined in \eqref{e-quasimonomial} where 
$\alpha=(\frac{a_1}{A_{X,\Delta}(E_1)},..., \frac{a_r}{A_{X,\Delta}(E_r)}).$
 \end{defn}
 
 Then the cone $i_{X,\Delta}\big(\cD(E)\big)\times \mathbb{R}_{>0}\subset \Val_{X,x}$ consists of all valuations $v$ such that $\lambda\cdot v\in i_{X,\Delta}(\cD(E))$ for some rescaling $\lambda\in \mathbb{R}_{>0}$.

\subsection{Log canonical thresholds and Koll\'ar components}

\subsubsection{Log canonical thresholds}

\begin{defn}Given a non-zero ideal $\fa$ on a log canonical pair $(X,\Delta)$, we call $c$ to be the {\it log canonical threshold} $c:=\lct(X,\Delta;\fa)$ if 
$$c=\max_{t>0} \{\ t\ | \ (X,\Delta+t\cdot \fa) \mbox{ is log canonical, i.e. }A_{X,\Delta+t\cdot \fa}(v)\ge 0 \mbox{ for any }v \}$$
\end{defn}

We call any valuation $v$ satisfying that $A_{X,\Delta}(v)=c\cdot v(\fa)$ {\it a valuation that computes the log canonical threshold of $(X,\Delta)$ with respect to $\fa$}.

We have the following well known lemma. 

\begin{lem}\label{l-resolution}
Let $(X,\Delta)$ be a log canonical pair and $c=\lct(X,\Delta;\fa)$.   Let $\mu\colon Y\to (X,{\rm Supp}(\Delta)\cup Z(\fa))$ be a log resolution where $Z(\fa)$ is the subscheme defined by $\fa$. Denote by $\mu^*(\fa)=\mathcal{O}_Y(-E)$. Define $\Delta_Y$ by $K_Y+\Delta_Y:=\mu^*(K_X+\Delta)+cE.$ 

Then valuations $v$ that compute the log canonical threshold of $(X,\Delta)$ are precisely given by the points on the space
$$i_{X,\Delta}\big(\cD(\Gamma)\big)\times {\mathbb{R}_{>0}}\subset \Val_{X},$$ where $\Gamma\subset {\rm Supp}(\Delta_Y)$ consists of all the components in $\Delta_Y$ with coefficient 1.
\end{lem}
\begin{proof}The case when $v$ is a divisorial valuation follows from \cite[Corollary 2.31]{KollarMori98}. 

When $v$ is quasi-monomial, we can assume the model $Y_v$ in Definition \ref{d-logd}(2) is a log resolution of $(X,{\rm Supp}(\Delta)\cup Z(\fa))$. Let the center of $v$ be a generic point of the intersection of $\bigcap^r_{j=1} E_{j} $, and assume $v=v_{\alpha}$ where $\alpha=(\alpha_1,...,\alpha_r)$ with $\alpha_j>0$ for all $1\le j \le r$. By \eqref{e-logdis}, and $v_{\alpha}(\fa)=\sum^r_{j=1}{\alpha_{j}}\ord_{E_{j}}(\fa)$, we know that 
$$A_{X,\Delta}(v)=\sum^r_{j=1}\alpha_jA_{X,\Delta}(E_j)\ge \sum^r_{j=1}c\cdot\alpha_j \ord_{E_j}(\fa) =c\cdot v(\fa)$$ and the equality holds if and only if $A_{X,\Delta}(E_{j})=c\cdot \ord_{E_{j}}(\fa)$ for all $j$.

Finally, for a general valuation $v$, we consider the quasi-monomial valuation $\rho_Y(v)$. Then we know that 
$v(\fa)=v(E)=\rho_Y(v)(E)$, and $A_{X,\Delta}(\rho_Y(v))\le A_{X,\Delta}(v)$. Thus we know  $A_{X,\Delta}(\rho_Y(v))= A_{X,\Delta}(v)$ which implies $\rho_Y(v)=v$ (see \cite[Corollary 5.4]{JM-qmvaluations}). 
\end{proof}

For a graded sequence $\fa_{\bullet}=\{\fa_k\}_{k\in \bN}$ of non-zero ideals on a klt pair, we can also define its log canonical threshold 
$$\lct(X,\Delta; \fa_{\bullet}):=\limsup_k \lct(X,\Delta; \frac{1}{k}\fa_k)\in [0,+\infty]. $$
It was shown in \cite[Corollary 6.9]{JM-qmvaluations} that 
$$\lct(X,\Delta; \fa_{\bullet})=\inf_{w\in \Val_{X}}\frac{A_{X,\Delta}(w)}{w(\fa_{\bullet})}.$$
Moreover, by \cite[Theorem A and Theorem 7.3]{JM-qmvaluations},  if $c:=\lct(X,\Delta; \fa_{\bullet})<+\infty$, then there always exists a valuation $v$ satisfying $A_{X,\Delta}(v)=c\cdot v(\fa_{\bullet})$, i.e.,
$$\frac{A_{X,\Delta}(v)}{v(\fa_{\bullet})}=\inf_{w\in \Val_X} \frac{A_{X,\Delta}(w)}{w(\fa_{\bullet})} .$$
However, if $\fa_{\bullet}$ is not finitely generated, we usually can not expect that the log canonical threshold of $\fa_{\bullet}$ is computed by a divisorial valuation $v$ (see e.g. \cite[Example 8.5]{JM-qmvaluations}).

\begin{lem}[{\cite[Proposition 7.10]{JM-qmvaluations}}]\label{l-enlarge}
Let $(X,\Delta)$ be a klt pair, and $\fa_{\bullet}$ a graded sequence of  ideals with $\lct(X,\Delta; \fa_{\bullet})<\infty$. Assume a valuation $v$ computes its log canonical threshold with $v(\fa_{\bullet})=1$. Then if we let $\fa^v_{\bullet}=\{\fa^v_k\}_{k\in \bN}$ be the graded sequence of valuation ideals associated to $v$,  any valuation which computes the log canonical threshold of $\fa^v_{\bullet}$ must also compute the log canonical threshold of $\fa_{\bullet}$.
\end{lem}
\begin{proof}Since $1=v(\fa_{\bullet})=\inf \frac{1}{k}v(\fa_k)$, we know that $\fa_{k}\subset \fa^v_{k}$. Thus, for any $w\in \Val_{X,x}$, since $w(\fa_{\bullet})\ge w(\fa^v_{\bullet})$, we have
$$A_{X,\Delta}(v)\le \frac{A_{X,\Delta}(w)}{w(\fa_{\bullet})}\le \frac{A_{X,\Delta}(w)}{w(\fa^v_{\bullet})}. $$
Since $v(\fa^v_{\bullet})=1$, this implies that $v$ also computes the log canonical threshold of $\fa^v_{\bullet}$.

Moreover, for any $v'$ computes the log canonical threshold of $\fa^v_{\bullet}$, we may assume $v'(\fa^v_{\bullet})=1$, thus $A_{X,\Delta}(v)=A_{X,\Delta}(v')$, and since 
$v'(\fa^v_{\bullet})\le v(\fa_{\bullet})$, we conclude that  $v'(\fa_{\bullet})=1$ and $v'$ also computes the log canonical threshold of $\fa_{\bullet}$.
\end{proof}

\subsubsection{Koll\'ar components}

The following special type of valuations  will play a central role in our work. 

\begin{defn}[Koll\'ar Components]\label{d-kollarcomp}
Let $x\in (X,\Delta)$ be a klt singularity. A prime divisor $S$ over $(X,\Delta)$ is a {\it Koll\'ar component} if there is a birational morphism $\mu\colon Y\to X$ of $(X,\Delta)$ such that $\mu$ is an isomorphism over $X\setminus \{x\}$, $\mu^{-1}(x)=S$ is $\bQ$-Cartier and if we write
$$(K_Y+\mu_*^{-1}\Delta+S)|_S=K_S+\Delta_S,$$
(see \cite[Definition 4.2]{Kollar13} for the meaning), then $(S,\Delta_S)$ is a log Fano pair. 

By inversion of adjunction, this is equivalent to saying that $(Y,\mu^{-1}_*\Delta+S)$ is plt and $K_Y+\mu^{-1}_*\Delta+S\sim_{X,\bQ}A_{X,\Delta}(S)\cdot S$ is anti-ample over $X$. 
\end{defn}
Such a morphism $\mu$ is called  {\it a plt blow up} (see \cite{Prok-plt}), as $(Y,\mu_*^{-1}\Delta+S)$ is a plt pair. Inspired by the construction in \cite{Xu-kollar} and  the special test configuration construction in \cite{LX-specialTC} in the global setting, Koll\'ar components were systematically used to study various functions on the space of valuations in the local setting in \cite{LX-SDCI}.
\begin{lem}\label{l-kollarcompute}
Let $x\in (X,\Delta)$ be a klt singularity, and $c=\lct(X,\Delta;\fa)$ from some $\fm_x$-primary ideal $\fa$. 
Then there exists a Koll\'ar component over $x\in (X,\Delta)$ which computes the log canonical thresholds of $(X,\Delta)$ with respect to $\fa$.
\end{lem}
\begin{proof}See \cite[Propsition 2.10]{LX-SDCI}.
\end{proof}

For a graded sequence $\fa_{\bullet}=\{\fa_{k}\}_{k\in \bN}$ of ideals with a finite log canonical threshold, we have an approximation type result by Koll\'ar components, see Proposition \ref{p-approxlc}. 

\subsection{Family of pairs}

Let $X\to B$ be a flat family with geometric integral fibers, and $D$ a family of Weil divisors on $X$, i.e. ${\rm Supp}(D)$ does not contain any $X_s$. Then we call  $(Y,E)/B\to (X,D)/B$ to be {\it a fiberwise log resolution} of $(X,D)/ B$ where $E$ is the sum of the birational transform of $D$ and the exceptional divisor ${\rm Ex}(Y/X)$, if for each $s\in B$, $(Y_s,E_s)\to (X_s,D_s)$ is a log resolution and any strata of $(Y,E)$, i.e., a component of the intersection $\cap E_i$ for components $E_i$ of $E$, has geometric irreducible fibers over $B$.

\begin{defn-lem}\label{l-logresol}
Let $X$ be a variety over a finite type base $B$ with geometrically integral fibers, and $D\subset X$ a codimension one subvariety which does not contain any fiber of $X\to B$. 
Then we can stratify $B$ into a union of finitely many constructible subsets $B'_{\alpha}$, such that over each $B'_{\alpha}$, $(X, D)\times_B B'_{\alpha}$ admits a log resolution $\mu_{\alpha}\colon Y_{\alpha}\to (X, D)\times_B B'_{\alpha}$ i.e. $(Y_{\alpha},E_{\alpha}:={\rm Ex}(\mu_{\alpha})+\mu_{\alpha*}^{-1}D)$ is simple normal crossing, and each stratum is log smooth over $B'_{\alpha}$. In particular, for each $s\in B'_{\alpha}$, $(Y_s, E_s)\to (X_s,D_s)$ is a log resolution.

Moreover, we can replace $B'_{\alpha}$ by a finite \'etale cover $B_{\alpha}$ such that for each irreducible stratum $Z$ of $(Y_\alpha,E_{\alpha})$, the fibers of $Z\to B_{\alpha}$ are also irreducible. So $(Y_\alpha,E_{\alpha})/B_{\alpha}$ is  {\it a fiberwise log resolution} of $(X,D)\times_B B_{\alpha}$.  
\end{defn-lem}

\begin{defn}[$\mathbb{Q}$-Gorenstein family of klt pairs]\label{d-familypair}
We call $(X,\Delta)\to B$ to be {\it a $\mathbb{Q}$-Gorenstein family of klt pairs} over a smooth base $B$, if 
\begin{enumerate}
\item $X$ is flat over $B$ and $K_{X/B}+\Delta$ is $\bQ$-Cartier;
\item for any $s\in B$,  $X_s$ is normal and ${\rm Supp}(\Delta)$ does not contain $X_s$; and
\item for any $s\in B$, the pair $(X_s,\Delta_s)$ is klt, where $\Delta_s$ is the cycle theoretic restriction over $s\in B$. 
\end{enumerate}

We call $B\subset (X,\Delta)\to B$ {\it a $\mathbb{Q}$-Gorenstein family of klt singularities} over a smooth base $B$, if $(X,\Delta)\to B$ is a family of klt pairs over $B$, and there is a section $B\subset X$. We call $(X,\Delta)\to B$ to be {\it a $\mathbb{Q}$-Gorenstein family of log Fano pairs} over a smooth base $B$, if $(X,\Delta)\to B$ is a projective $\bQ$-Gorenstein family of klt pairs over $B$, and the fiber $(X_s,\Delta_s)$ is log Fano for any $s\in B$. 
\end{defn}
\begin{rem}Over a general (possibly singular) base, a correct definition of a family of klt pairs is subtle. See \cite{Kollar-modulibook} for a systematic study on this topic. For simplicity, in this note, we mostly only work over the smooth base.  We could allow a more general base $B$. In fact, what is needed is that for any `{\it admissible}' morphism $B'\to B$, we have a compatible definition of the pull back of the family. Such a theory is worked out whenever $B$ and $B'$ are reduced in \cite[Chapter 4]{Kollar-modulibook}. Using it, our results in this note can be extended to the case that when $B$ is reduced. 

When $\Delta=0$, we can even work over non-reduced base as below. 
\end{rem}

\begin{defn}[Locally stable family of klt varieties]\label{d-kollar}
We call $X\to B$ to be {\it a locally stable family of klt varieties} over a finite type base scheme $B$, if 
\begin{enumerate}
\item $X$ is flat over $B$ and for any $m$, $\omega^{[m]}_{X/B}$ is flat over $B$ and commutes with any base change $B'\to B$, 
\item for any point $s\in B$,  $X_s$ is klt.
\end{enumerate}
See \cite[Chapter 3]{Kollar-modulibook} for more background. 
\end{defn}

\begin{lem}\label{l-lctfamily} 
Let $(X,\Delta)\to B$ be a $\bQ$-Gorenstein family of klt pairs over a smooth base $B$, and $D\subset X\times_B U\to U$ a family of effective Cartier divisors on $X$ over a finite type variety $\pi\colon U\to B$. Fix a constant $c>0$. 

There is a constructible set $V\subset U$, such that if we denote by $D_u$ the divisor corresponding to a point $u\to U$  and $(X_u,\Delta_u)=(X,\Delta)\times_B \{u\}$, then $\lct(X_u,\Delta_u; D_u)=c$  if and only if $u$ factors through $ V$.  
\end{lem}
\begin{proof}See \cite[Section 9.5.D]{Laz-positivityII} or \ref{p-familyresol}.
\end{proof}

\begin{say}\label{p-familyresol}
If we apply Definition-Lemma \ref{l-logresol} to the setting of Lemma \ref{l-lctfamily}, we can  stratify $V$ into a union of finitely many constructible subsets and take finite \'etale coverings to get finitely many varieties $\{V_{\alpha}\}$, such that 
$\bigsqcup_{\alpha}{V_{\alpha}}\to V$ is surjective,  and for any $\alpha$
$$(X\times_B V_{\alpha},\Delta\times_B V_{\alpha}+c\cdot D\times_U V_{\alpha})$$
admits a fiberwise log resolution $(Y_\alpha,E_{\alpha})/V_{\alpha}$. Let $(E_{\alpha})_j$ ($j=1,...,r$) on $Y_{\alpha}$ be the log canonical places over  $(X\times_B V_{\alpha},\Delta\times_B V_{\alpha}+c\cdot D\times_U V_{\alpha})$. Their reductions $(E_{\alpha})_{u,j}$ over any point $u\in V_{\alpha}$ give precisely all divisors on $Y_u$ which are log canonical places over $(X_s, \Delta_s+cD_u)$ where $s=\pi(u)$. So for any $\alpha$, we can identify the dual complexes 
\begin{eqnarray}\label{e-dualcomplexes}
\cD(\Gamma_{\alpha}:=\sum^r_{j=1}(E_{\alpha})_j)\mbox{\ \  and\ \  }\cD(\Gamma_{\alpha,u}:=\sum^r_{j=1}(E_{\alpha})_{u,j}) 
\end{eqnarray}
for any $u\in V_{\alpha}$.
\end{say}

\subsection{Boundedness of complements}The concept of {\it complement} was an idea first introduced in \cite{Shokurov-threedim} to understand morphisms with a relative anti-ample canonical bundle. At the first sight, it seems to be technical. However,  the boundedness of complement proved in \cite{Bir-BABI} is a major step forward to study birational geometry of Fano varieties. For this note, we need the following local result.

\begin{thm}[{\cite[Theorem 1.8]{Bir-BABI}}]\label{t-localcomplement}
Fix a positive integer $n$ and a finite rational set $I\subset [0,1]\cap \mathbb{Q}$. Then there exists a positive integer $N_0=N_0(n,I)$ depending only on $n$ and $I$, such that for any klt singularity $x\in (X,\Delta)$ with $\dim(X)=n$ and the coefficients of $\Delta$ contained in $I$, if there is a Koll\'ar component $S$ given by the exceptional divisor of the plt blow up $\mu \colon Y\to (X,\Delta)$, then there is a divisor $\Delta^+\ge \Delta$ which satisfies that $(X,\Delta^+)$ is log canonical, $N_0(K_X+\Delta^+)\sim 0$ and $S$ is a log canonical place of $(X,\Delta^+)$. 
\end{thm}
\begin{proof}Denote by $\Delta_Y:=\mu^{-1}_*\Delta$. 
By \cite[Theorem 1.8]{Bir-BABI}, there is a constant $N_0=N_0(n,I)$ which only depends on $n$ and $I$, and  a $\mathbb{Q}$-divisor $\Theta \ge 0$ such that $(Y,\Theta+\Delta_Y+S)$ is log canonical, and 
$$N_0(K_{Y}+\Theta+\Delta_Y+S)\sim_{X} 0.$$

Push forward to $X$ and denote by $\Psi:=\mu_{*}(\Theta)$ and $\Delta^+:= \Delta+\Psi$. Since 
$$\mu^*(K_X+\Delta^+)=K_Y+\Delta_Y+S+\Theta$$
 we know that $(X, \Delta^+)$ is log canonical with $S$ being a log canonical place. Moreover, $N_0(K_X+\Delta^+)$ is Cartier. 

\end{proof}

\subsection{Local volumes}

\subsubsection{Definitions}
For a valuation $v$ centered on a klt singularity $x\in (X,\Delta)$, we give the definitions of two volume functions defined on $\Val_{X,x}$, namely the volume $\vol_{X,x}(v)$ (or $\vol(v)$)  and the normalized volume $\hvol_{(X,\Delta),x}(v)$ (or abbreviated as $\hvol_{X,\Delta}(v)$ or simply $\hvol(v)$ if there is no confusion).

\begin{defn}\label{d-vol}
 Let $X$ be an $n$-dimensional normal variety. Let $x\in X$ be a closed point. We define
 the \emph{volume of a valuation} $v\in\Val_{X,x}$ following \cite{ELS-localvolume} as
 \[
  \vol_{X,x}(v)=\limsup_{k\to\infty}\frac{\ell(\cO_{x,X}/\fa^v_k)}{k^n/n!}.
 \]
 where $\ell$ denotes the length of the artinian module.
\end{defn}
Thanks to the works of \cite{ELS-localvolume, LM-okounkovbody, Cutkosky-multiplicities} the above limsup is actually a limit.

\medskip

The following invariant, which was defined first in \cite{Li-minimizer}, plays a key role for our study in the local stability. 

\begin{defn}[\cite{Li-minimizer}]\label{d-normvol}
 Let $(X,\Delta)$ be an $n$-dimensional klt log pair. Let $x\in X$ be a closed point.
 Then the \emph{normalized volume function of valuations} $\hvol_{(X,\Delta),x}:\Val_{X,x}\to(0,+\infty)$
 is defined as
 \[
  \hvol_{(X,\Delta),x}(v)=\begin{cases}
            A_{X,\Delta}(v)^n\cdot\vol_{X,x}(v), & \textrm{ if }A_{X,\Delta}(v)<+\infty;\\
            +\infty, & \textrm{ if }A_{X,\Delta}(v)=+\infty.
           \end{cases}
 \]
 The \emph{volume of the klt singularity} $(x\in (X,\Delta))$ is defined as
 \[
  \hvol(x, X,\Delta):=\inf_{v\in\Val_{X,x}}\hvol_{(X,\Delta),x}(v).
 \]
 
 For a divisorial valuation $\ord_E$, we will also use $\hvol(E)$ for $\hvol(\ord_E)$.  
 
 It is known the minimum indeed exists by \cite{Blu-Existence} (see also Remark \ref{r-existence}).
 
\end{defn}

The minimizing problem for $\hvol_{X,\Delta}$ is closely related to K-stability. The guiding question is called the {\it Stable Degeneration Conjecture}, which was formulated in \cite[Conjecture 7.1]{Li-minimizer} and \cite[Conjecture 1.2]{LX-SDCII}. See \cite{LLX-Tiansurvey} for more background. Theorem \ref{t-minimizing} settles one part of the conjecture. 

We need the following result which is a special case of the Stable Degeneration Conjecture.

\begin{prop}\label{p-cone}
Let $(V,\Delta_V)$ be a log Fano pair. Let $r$ be a positive integer such that $H:=-r(K_V+\Delta_V)$ is Cartier. Consider the cone $x\in (X,\Delta)=C(V,\Delta_V;H)$, with $x$ being the vertex. Let $v^*\in \Val_{X,x}$ be the canonical divisorial valuation obtained by blowing up the vertex. Then $v^*$ is a minimizer of $\hvol_{(X,\Delta),x}$ if and only if $(V,\Delta_V)$ is K-semistable. 
\end{prop}
\begin{proof}This was proved in \cite[Theorem 4.5]{LX-SDCI}, after the works in \cite{Li-K=M} and \cite{LL-K=M}.
\end{proof}

\subsubsection{Invariance of local volumes}
The following theorem is a local version of \cite[Theorem 4.2]{HMX-auto} and the proof is similar to the one there.
\begin{thm}\label{t-invlocal}
Let $B\subset(X,\Delta)\to B$ be a $\mathbb{Q}$-Gorenstein family of klt pairs over a smooth base $B$. 
Assume there is a fiberwise log resolutions $\mu\colon Y\to (X,\Delta)$ over $B$  (see Definition-Lemma \ref{l-logresol}) with the exceptional divisor $E=\sum^k_{i=1} E_i$. If $F$ is a prime toroidal divisor with respect to $(Y,{\rm Supp}(\mu_*^{-1}(\Delta))+E )$, with $A_{X,\Delta}(F)<1$, then the volume $\vol_{X_s,\Delta_s}(\ord_{F_s})$ is locally constant on $s\in B$. 
\end{thm}
\begin{proof}By restricting over a curve $C\to B$, we can assume $B$ is a smooth curve. By taking a toroidal resolution, we can assume $F$ is a divisor on $Y$. Write 
$$\mu^*(K_{X}+\Delta)=K_{Y}+F_1-F_2,$$ where $F_1$ and $F_2$ are effective $\bQ$-divisors without any common components. By our assumption, each stratum of $(Y, {\rm Supp}(F_1+F_2))$ is smooth over $B$ and $F\subset {\rm Supp}(F_1)$. After possibly a further toroidal blow up, we may assume $(Y, F_1)$ is terminal. 

Fix $s$, for a sufficiently small $\epsilon\in \bQ_{>0}$ satisfying $\epsilon<\mult_FF_1$, the divisor $N_{\sigma}(Y_s/X_s; K_{Y_s} +(F_1)_s-\epsilon F_s)$ defined as \cite[III.4]{Naka-zariski} is a $\mathbb{Q}$-divisor, as $({Y_s} ,(F_1)_s-\epsilon F_s)$ has a relative good minimal model over $X_s$. In particular,
$$\Gamma_s :=\Big((F_1)_s-\epsilon F_s\Big) -\Big(((F_1)_s-\epsilon F_s)  \wedge N_{\sigma}\big(Y_s/X_s;K_{Y_s}+(F_1)_s-\epsilon F_s\big)\Big)$$
is also a $\mathbb{Q}$-divisor. Therefore, we can choose a divisor $\Gamma$ supported on ${\rm Supp}(F_1)$ such that $\Gamma|_{Y_s}=\Gamma_s$.

 We run a relative MMP program with scaling by \cite{BCHM} for
$K_{Y}+\Gamma$ over $X$.
Denote by $g^k\colon Y^k\dasharrow Y^{k+1}$ the $k$-th MMP step, and $\Gamma^k$ the pushforward of $\Gamma$ to $Y^k$. 
We will inductively prove  that,
\begin{enumerate}
\item[(a)] $g^k$ is isomorphic at the generic point of every component of $\Gamma|_{s}$; and
\item[(b)] $g_s^k\colon Y_s^k\dasharrow Y_s^{k+1}$ is a birational contraction.
\end{enumerate}

Assume this is true after $(k-1)$-steps. $(b)_{k-1}$ implies that no component of $\Gamma^k_s$
is a component of the stable base locus of $K_{Y^k_s}+\Gamma^k_s$. Then if $g^k$ is not an isomorphism at the generic point of a divisor $D$ contained in $Y^k_s$, $D$ is covered by curves $C$ such that
$$0>C\cdot(K_{Y^k}+\Gamma^k)=C\cdot(K_{Y^k_s}+\Gamma_s^k),$$
It follows that $D$ is a component of the stable base locus of $K_{Y^k_s}+\Gamma^k_s$, thus $D$ is not a component of $\Gamma^k_s$. This is $(a)_k$.

To see $(b)_k$,  since the MMP is also a $(K_{Y_s^k}+\Gamma_s^k )$-negative morphism, if $g_s^k\colon Y_s^k\dasharrow Y_s^{k+1}$ is not a birational contraction, a component $\Theta$ in ${\rm Ex}((g_s^k)^{-1})$ will have non-positive discrepancy for $(Y_s^{k+1},\Gamma^{k+1}_{s} )$, thus it has a negative discrepancy with respect to $(Y_s^{k},\Gamma_s^k )$. But as $(Y^k_s, \Gamma_s^k)$ is terminal, and each step is $(K_{Y^k_s}+\Gamma_s^k)$-negative, we know $\Theta$ is component of ${\rm Supp}(\Gamma^k_s)$ which is a contradiction to $(a)_k$.

By \cite{BCHM}, we obtain a relative minimal model $\phi\colon Y\dasharrow Z$  and $ \psi\colon Z \to X $. For $m$ sufficiently divisible,
\begin{eqnarray*}
(\mu_s)_*\cO_{Y_s}(-m\epsilon{F_s})&=& (\mu_s)_*\cO_{Y_s}\big(-m(\epsilon F_s-(F_2)_s)\big)\\
&\cong& (\mu_s)_*\cO_{Y_s}\big(m(K_{Y_s}+(F_1)_s-\epsilon F_s)\big) \\
&=&(\mu_s)_*\cO_{Y_s}\big(m(K_{Y_s}+\Gamma_s)\big)\\
&=& (\psi_s)_*\cO_{Z_s}\big(m(K_{Z_s}+(\phi_{s})_*\Gamma_s)\big)\\
&\cong&\psi_*\cO_{Z}\big(m(K_{Z}+\phi_*\Gamma)\big) \cdot  \mathcal{O}_{X_s}\\
&\subset & \mu_*\cO_Y\big(m(K_{Y}+(F_1)-\epsilon F)\big) \cdot \mathcal{O}_{X_s} \\
&\cong & \mu_*\cO_Y\big(m(F_2-\epsilon F)\big)\cdot \mathcal{O}_{X_s} \\
&=&\mu_*\cO_Y(-m\epsilon {F})\cdot\mathcal{O}_{X_s},
 \end{eqnarray*}
where in the fourth line, we use  $(b)$ and that $(Y_s,\Gamma_s)\dasharrow (Z_s, (\phi_s)_*\Gamma_s)$ is $(K_{Y_s}+\Gamma_s)$-negative; in the fifth line, we use $\phi\colon Z\to X $ is a relative minimal model of $(Z,\phi_*\Gamma)$ and $\phi_*(\Gamma)|_{Z_s}=(\phi_s)_*(\Gamma_s)$ by $(a)$ and $(b)$.

 So $\mu_*\cO_Y(-m{F}) \cdot  \mathcal{O}_{X_s}\cong (\mu_s)_*\cO_{Y_s}(-m{F_s})$ for any $s$ and sufficiently divisible $m$, which implies 
 $$\frac{\cO_X}{\mu_*\cO_Y(-m{F})} \cdot {\cO_{X_s}} \cong \frac{\cO_{X_s}}{(\mu_s)_*\cO_{Y_s}(-m{F_s})}.$$
 Thus we conclude that the finite rank $\cO_B$-module $\cO_X/\mu_*\cO_Y(-m{F})$ is locally free, whose restriction over $\Spec k(s)$ is isomorphic to ${\cO_{X_s}}/{(\mu_s)_*\cO_{Y_s}(-m{F_s})}$. As a result, we know that
 \[
 \vol(\ord_{F_\eta})=\lim_{m\to \infty}\frac{{\rm rk}\big(\cO_X/\mu_*\cO_Y(-m{F})\big)}{m^n/n!}=\lim_{m\to \infty}\frac{\dim \big(\cO_{X_s}/\mu_*\cO_{Y_s}(-m{F_s})\big)}{m^n/n!}=\vol(\ord_{F_{s}})
 \]
 is a constant.
\end{proof}
The following corollary may be of independent  interest. We will not need it in the rest of the paper.
\begin{cor}In the notation of Theorem \ref{t-invlocal}. Let $\chi\colon Z^{c}\to X$ be the birational model which precisely extracts the birational transform of $F$, denoted by $F'$, such that $-F'$ is ample over $X$. Then it satisfies that restricting over each $s\in B$, $\chi_s\colon Z^{c}_s\to X_s$ precisely extracts $F'_s$ which is the birational transform of $F_s$.
\end{cor}
\begin{proof}In the proof of Theorem \ref{t-invlocal}, we have shown under the assumption there, 
$$\mu_*\cO_Y(-mF)\cdot k_s\cong (\mu_s)_*\cO_{Y_s}(-mF_s).$$ 

For any sufficiently divisible $m$, $\chi\colon Z^c\to X$ (resp. $\chi_s\colon Z^c_s\to X_s$) is indeed the blow up of the ideal sheaf $\mu_*\cO_Y(-mF)\subset \cO_X$ (resp.  $(\mu_s)_*\cO_{Y_s}(-mF_s)\subset \cO_{X_s}$). Thus by the above isomorphism, $Z^c_s$ is the birational transform of $X_s$ under the morphism $\chi$. Since $\mu_*\cO_Y(-mF)$ is flat over $B$, we know $Z^c$ is flat over $B$, therefore $Z^c\times_B \{s\}$ coincides with $Z^c_s$. 
\end{proof}

\section{Quasi-monomial limit}
\subsection{Approximation}
On a klt singularity $x\in (X,\Delta)$, for a graded sequence $\fa_{\bullet}=\{\fa_k\}_{k\in \bN}$ of $\fm_x$-primary ideals, unlike Lemma \ref{l-kollarcompute}, usually we can not find a divisorial valuation computing its log canonical threshold. However, we have the following result, whose proof slightly simplifies the one in {\cite[Theorem 1.3]{LX-SDCI}}. 

\begin{prop}\label{p-approxlc}
Let $x\in (X,\Delta)$ be a klt singularity. 
Let $\fa_{\bullet}=\{\fa_k\}_{k\in\bN}$ be a graded sequence of $\fm_x$-primary ideals  with $\lct(X,\Delta,\fa_{\bullet})<+\infty$, then we can find a valuation $v\in \Val^{=1}_{X,x}$ which is the limit of $\frac1{A_{X,\Delta}(S_j)} \cdot \ord_{S_j}$ for a sequence of Koll\'ar components $\{S_j\}$, such that 
 $v$ calculates the log canonical threshold of $\fa_{\bullet}$. 
\end{prop}
Later in Theorem \ref{t-quasi-monomial}, we will show such $v$ is always quasi-monomial.
\begin{proof}Let $c=\lct(X,\Delta,\fa_{\bullet})$ and $w\in \Val_X^{=1}$ a valuation which calculates the log canonical threshold of $\fa_{\bullet}$, then $w(\fa_{\bullet})=\frac{1}{c}$.
Let $\fa_k$ $(k\in \bN)$ be the $k$-th element in the graded sequence of ideals.  Denote by
$c_k:=\lct(X,\Delta; \frac{1}{k}\fa_k)$.
In particular, $\lim_k  c_k=c. $

By Lemma \ref{l-kollarcompute}, there exists a Koll\'ar component $S_k$ with 
$c_k\cdot \ord_{S_k}(\fa_k)=k\cdot A_{X,\Delta}(S_k).$
We consider the valuation 
$$v_k:=\frac{c_k}{c\cdot A_{X,\Delta}(S_k)}\ord_{S_k}=\frac{k}{c\cdot\ord_{S_k}(\fa_k)}\ord_{S_k}.$$ 
Note that $A_{X,\Delta}(v_k)=\frac{c_k}{c}\le 1$.

Assume $\fm_x^p\subset \fa_1$ for some $p\gg 0$, as $\fa_1$ is $\fm_x$-primary. Then $\fm_x^{pk}\subset \fa_1^k\subset \fa_k$.
Thus for any $k$,
$$v_k(\fm_x)\ge v_k(\fa_k)\cdot \frac{1}{pk}=\frac{1}{cp},$$
which is bounded from below. In particular, by the compactness result \cite[Proposition 5.9]{JM-qmvaluations} and \cite[Proposition 3.9]{LX-SDCI}, we know that there is an infinite sequence $\{v_{j}\}$ which has a limit in $\Val_{X,x}$,  denoted by
$v=\lim_{j\to \infty} v_{j}.$

We have
$$A_{X,\Delta}(v)\le \liminf_{j\to \infty} A_{X,\Delta}(v_{j})\le 1,$$ as $A_{X,\Delta}$ is lower semicontinuous (see \cite[Lemma 5.7]{JM-qmvaluations}). 
By definition $v_k(\fa_k)=\frac{k}{c}$ for any $k$ and $\fa^m_{k}\subset \fa_{mk}$ for any $m\in \bN$. This implies 
$$v_{mk}(\fa_k)\ge \frac{v_{mk}(\fa_{mk})}{m}=\frac{k}{c}.$$

Thus $$v(\fa_{\bullet})=\lim_{k \to \infty} \frac{v(\fa_k)}{k}=\lim_{k \to \infty}\big(\lim_{m\to \infty}\frac{v_{mk}(\fa_k)}{k}\big)\ge \frac{1}{c}.$$ Since 
$$\frac{A_{X,\Delta}(v)}{v(\fa_{\bullet})}\le c= \inf_{w'} \frac{A_{X,\Delta}(w')}{w'(\fa_{\bullet})}.$$
This implies that $A_{X,\Delta}(v)=1$ and $v(\fa_{\bullet})=\frac{1}{c}$. 

Since $\lim_{j} c_{j}=c$, we have
$$v=\lim_{j} v_{j}=\lim_{j} \frac{c}{ c_{j}}v_{j}=\lim_{j} \frac{1}{A_{X,\Delta}(S_{j})}\ord_{S_{j}}.$$
\end{proof}

We can slightly improve the result by showing the following. 

\begin{lem}\label{l-valuation}
In the notation of Proposition \ref{p-approxlc}. Assume $w\in \Val^{=1}_{X,x}$ which calculates the log canonical threshold of $\fa_{\bullet}$. Then we can choose $v$ as in Proposition \ref{p-approxlc} such that $v\ge w$.
\end{lem}
\begin{proof}Let $c=\lct(X,\Delta,\fa_{\bullet})$ and then $w(\fa_{\bullet})=\frac{1}{c}$. Let $w'=c\cdot w$, by Lemma \ref{l-enlarge}, we can assume $\fa_{\bullet}=\fa^{w'}_{\bullet}$. Then we can apply the construction in the proof of Proposition \ref{p-approxlc} to get $v$. It remains to show $v\ge w=\frac{1}{c}w'$. 
To verify it, we pick any $f\in R$ and denote by $w'(f)= p$ for some $p\in \bR_{>0}$. For a fixed $j$, choose $l$ such that 
$$(l-1)p< j\le lp .$$ Let $k=j$ in the previous construction. Then we have: 
\begin{eqnarray*}
 w'(f)= p&\Longrightarrow & w'(f^l)= pl,\\
&\Longrightarrow & f^l\in \fa_{pl},\\
&\Longrightarrow & f^l\in \fa_{j},\\
&\Longrightarrow&v_{j}(f)\ge \frac{j}{cl}>\frac{p}{c}-\frac{p}{cl}.
\end{eqnarray*}
The fourth arrow is because $v_{j}(\fa_{j})=\frac{j}{c}$. 
Thus 
$$v(f)=\lim_j v_{j}(f)\ge \frac{p}{c}= w(f).$$
\end{proof}

\subsection{Quasi-monomial limit}
Let $x\in (X={\rm Spec}(R),\Delta)$ be a klt singularity. The main aim of this section is to prove the following theorem.
\begin{thm}\label{t-quasi-monomial}Let $x\in (X,\Delta)$ be a klt singularity. Let $v_i=\frac{1}{A_{X,\Delta}(S_i)}(\ord_{S_i})$ be an infinite sequence of valuations  where $S_i$ are Koll\'ar components over $x\in (X,\Delta)$ and assume there is a uniform $C$ and $\delta>0$ such that either $\hvol(v_i)< C$ or $v_i(\fm_x)>\delta$, then there is an infinite subsequence which has a quasi-monomial limit  $v\in \Val_{X,x}^{=1}$.
\end{thm}

\begin{lem}\label{l-order}Fix  constants $C$ and  $\epsilon>0$. 
If there is a sequence of valuations $\{v_i\}_{i\in \bN}$,  such that either
$$\big(\mbox{$\lim_{i\to \infty}v_i\to v\in \Val_{X,x}\big) \ \  $ or $\ \ \big( \hvol(v_i)<C$ and $A_{X,\Delta}(v_i)>\epsilon$ for all $i$$\big)$},$$ then there is a positive $\delta>0$, such that for any $i$, $v_i(\fm_x)\ge \delta $. 
\end{lem}
\begin{proof}Let us first assume $\lim_{i\to \infty}v_i\to v\in \Val_{X,x}$.
Let $f_1,...,f_k$ be a set of generators of $\fm_x$. Then for any $1\le j \le k$, $\lim_{i\to \infty} v_i(f_j)\to v(f_j)>0$. In particular, we know for all $1\le j \le k$ and all $i$,
$v_i(f_j)$ has a positive lower bound, which we denote by $\delta$. Then $v_i(\fm_x)\ge \delta$ for any $i$. 

If we assume $\hvol(v_i)<C$ and $A_{X,\Delta}(v_i)>\epsilon$, then this follows from \cite[Theorem 1.1]{Li-minimizer}.
\end{proof}
 
\begin{prop}\label{p-complement}
Let $(X,\Delta)$ be a klt pair. Let $\{S_i\}_{i\in \bN}$ be a sequence of Koll\'ar components such that 
$$\lim_{i\to \infty}\frac{1}{A_{X,\Delta}(S_i)}\ord_{S_i}=v.$$
Then there exists a constant $N$ and a family of Cartier divisors $D\subset X\times V$ parametrised by a variety $V$ of finite type, such that for any $u\in V$, $(X,\Delta+\frac{1}{N}D_u)$ is lc but not klt; and for any $i$, $S_i$ computes the log canonical threshold of a pair $(X,\Delta+\frac{1}{N}D_{u_i})$ for some $u_i\in V$. 
\end{prop}
For a stronger statement, which will be needed later, see Proposition \ref{p-familykollar}.
\begin{proof}Denote by $v_i:=\frac{1}{A_{X,\Delta}(S_i)}\cdot \ord_{S_i}$.  Let $\mu_i\colon Y_i\to X$ be the plt blow up which extracts $S_i$. 

By Theorem \ref{t-localcomplement}, We know that there is a uniform $N_0$ such that for each $i$, we can find an effective $\mathbb{Q}$-divisor $\Psi_i$ with the property that $(X, \Delta+\Psi_i)$ is log canonical with $S_i$ being a log canonical place. 
Define $\Delta^+_{S_i}$ by
$$K_{S_i}+\Delta^+_{S_i}=\mu_i^*(K_X+\Delta+\Psi_i)|_{S_i},$$ then $(S_i,\Delta^+_{S_i})$ is log canonical. 
Moreover, $N_0(K_X+\Delta+\Psi_i)$ is Cartier. Set $N=rN_0$ where $r$ is a positive integer such that $r(K_X+\Delta)$ is Cartier, then both $N(K_X+\Delta)$ and $N\Psi_i$ are Cartier for all $i$. Thus we can assume $N\Psi_i$ is given by ${\rm div}(\psi_i)$ for some regular function $\psi_i$.

Fix $M\in \bN$, such that $\delta\cdot M>N$ where $\delta$ is the positive constant obtained in Lemma \ref{l-order} for the sequence $\{v_i\}$. Then let $g_1,...,g_m$ be $m$-elements in $R$, such that their reductions 
 $$[g_1],..., [g_m]\in \mathcal{O}_{x,X}/\fm_x^M$$ yield a $k$-basis. So for any $i$, there exists a $k$-linear combination of $h_i$ of $ g_1,...,g_m$ such that the  image of $\psi_i$ and $h_i$ are the same in $\mathcal{O}_{x,X}/\fm_x^M$. 

\begin{claim}\label{claim-madic}
Let $\Phi_i:={\rm div}(h_i)$, then  $(X,\Delta+\frac{1}{N}\Phi_i)$ is log canonical and has $S_i$ as its log canonical place. 
\end{claim}
\begin{proof} Since $s_i=h_i-\psi_i\in \fm_x^M$, so $$v_i(s_i)\ge M\cdot v_i(\fm_x)> N.$$ On the other hand, since $v_i$ computes the log canonical threshold of $(X, \Delta+\Psi_i)$, we know 
$$N=N\cdot A_{X,\Delta}(v_i)=v_i(\psi_i)=v_i(h_i), $$ 
hence $A_{X,\Delta+\Psi_i}(v_i)=0$. This implies that
\begin{eqnarray*}
\big(K_{Y_i}+S_i+\mu^{-1}_{i*}(\frac{1}{N}\Phi_i+\Delta_i)\big)|_{S_i}&=& \mu_i^*(K_X+\Delta+\frac{1}{N}\Phi_i)|_{S_i}\\
&=&\mu_i^*(K_X+\Delta+\Psi_i)|_{S_i}\\
&=&K_{S_i}+\Delta^{+}_{S_i},
\end{eqnarray*}
(see  \cite[35]{Kollar-genericlimit} for the second equality).
Since $(S_i,\Delta^{+}_{S_i})$ is log canonical, by inversion of adjunction (see \cite{Kawakita-inversion}), we know that $({Y_i},S_i+\mu^{-1}_{i*}(\Delta+\frac{1}{N}\Phi_i))$ is log canonical along $S_i$, which implies that $(X,\Delta+\frac{1}{N}\Phi_i)$ is log canonical, 
and $S_i$ computes its log canonical threshold. 
\end{proof}
Then applying Lemma \ref{l-lctfamily} to the family of Cartier divisors $D_U\subset X\times U$, where 
$$U=\{(x_1,...,x_m)\in \bA_k^m\ |\ (x_1,...,x_m)\neq (0,...,0)\}$$
 and $D_U=(\sum^m_{i=1} x_ig_i=0),$ we could find a bounded family of divisors $D\subset X\times V\to V$ for some $V\subset U$, such that  $(X,\Delta+\frac{1}{N}D_u)$ is log canonical but not klt if and only if  $u\in V$. From our argument, we know $D\subset X\times V$ is the desired family of Cartier divisors. 
\end{proof}

The proof of of Claim \ref{claim-madic} says that the log canonical thresholds of two functions are the same if they are sufficiently close in $\fm_x$-adic topology.  As far as we know, this kind of argument first appeared in \cite{Kollar-genericlimit} and \cite{dFEM-ACC}. 

\begin{proof}[Proof of Theorem \ref{t-quasi-monomial}] By Lemma \ref{l-order}, we can assume $v_i(\fm_x)>\delta$. Since  $A_{X,\Delta}(v_i)= 1$ it follows from \cite[Prop 3.9]{LX-SDCI} that there is an infinite subsequence, which we still denote by $v_i$ such that $v:=\lim_i v_i$  exists. It remains to prove $v$ is quasi-monomial.

Applying Proposition \ref{p-complement}, we get a bounded family of Cartier divisors $(D\subset X\times V)\to V$ such that for any $u$, $(X,\Delta+\frac{1}{N}D_u) $ is log canonical but not klt, and any $S_i$ is the lc place of $(X,\Delta+\frac{1}{N}D_{u_i})$ for some $u_i\in V$.
  Replacing $V$ by an irreducible closed subset, we can further assume the set $\{u_i\}$ form a dense set of points on $V$. We may further resolve $V$ to be smooth. 

Applying  \eqref{p-familyresol} to $(X\times V,\Delta\times V+\frac{1}{N}{D})$ over $V$, after shrinking $V$ to an open set and replacing it by a finite \'etale covering, we can assume $(X\times V,\Delta\times V+\frac{1}{N}{D})\to V$ admits a fiberwise log resolution $\mu_V\colon {Y}\to (X\times V,\Delta\times V+\frac{1}{N}{D})$ over $V$ with the exceptional divisor $E=\sum^k_{i=1}E_i$ being simple normal crossing (see Definition-Lemma \ref{l-logresol}). We choose $\Gamma\subset E$ to be the subdivisor given by the components  with log discrepancy 0 with respect to $(X\times_B V, \Delta\times_B V+\frac{1}{N}D)$.
We also denote by $K:=K(V)$ the function field of $V$, and $\eta$ the point $x\times {\rm Spec}(K) \in X\times V$. 

For any point $u_i$, since $S_i$ computes the log canonical threshold of $(X,\Delta+\frac{1}{N} D_{u_i}(=\Phi_i))$, using the identification in \eqref{e-dualcomplexes}, $S_i\in i_{X_s,\Delta_s}(\cD(\Gamma|_{Y_{u_i}}))\times \bR_{>0}$ can be regarded as a restriction of a divisor corresponding to a point on $i_{X,\Delta}(\cD(\Gamma))\times \bR_{>0}$, whose restriction over the generic point ${\rm Spec}(K)$ of $V$ will yield a divisor, denoted by $T_i$. Thus, the valuations $\frac{1}{A_{X_K,\Delta_K}(T_i)}\ord_{T_i}$ are contained in the image of the fixed dual complex 
$i_{X_K,\Delta_K}(\cD(\Gamma|_{Y_K}))$  for all $i$.

Since $\cD(\Gamma)$ is compact, after passing to an infinite subsequence of $i$, the valuations $\frac{1}{A_{X_K,\Delta_K}(T_i)}\ord_{T_i}$ converges to a quasi-monomial valuation $w$ over $X_K$.

We claim that the restriction of $w$ to $K(X)\subset K(X_K)$ is $v$. In fact, if for any $f\in R$, we denote by $f_K$ its image under the injection $K(X)\subset K(X_K)$. Then Lemma \ref{l-valfamily} implies that after passing to an infinite subsequence, 
$$w(f_K)=\lim_{i\to \infty}\frac{1}{A_{X_K,\Delta_K}(T_i)}\ord_{T_i}(f_K)=\lim_{i\to \infty}\frac{1}{A_{X,\Delta}({S}_i)}\ord_{S_i}(f)=v(f).$$

By the Zariski-Abhyankar inequality (see \eqref{e-zainequ}), 
$${\rm rat. rk}(w)+{\rm tr. deg}(w)\le {\rm rat. rk}(v)+{\rm tr. deg}(v)+{\rm tr. deg}(K(X_K)/K(X)).$$
Since $w$ is Abhyankar, the left hand side is equal to $\dim(X)+\dim(V)$. Therefore, 
$${\rm rat. rk}(v)+{\rm tr. deg}(v)=\dim(X),$$ thus $v$ is an Abhyankar valuation on $K(X)$, which is the same as saying that it is a quasi-monomial valuation. 
\end{proof}

\begin{lem}\label{l-valfamily}The notation as above, for any $f \in R$,  $\ord_{T_i}(f_K)\le\ord_{S_i}(f)$, and the equality holds for infinitely many $i$.
\end{lem}
\begin{proof}The first inequality is straightforward. To see the equality, we can take a log resolution $W$ of $\big(Y, E+\mu_{V*}^{-1}(p_1^*({\rm div}(f)+\Delta )) \big)$ where $p_1\colon X\times V\to X$. There is an open set $V^{\circ}\subset V$, such that 
$$W\times_V V^{\circ}\to \big(Y, E+\mu_{V*}^{-1}(p_1^*({\rm div}(f)+\Delta )) \big)\times_V V^{\circ}\to V^{\circ}$$
could yield a fiberwise log resolution after a finite \'etale base change. Now it follows from \cite[Proof of Lemma 4.6]{JM-qmvaluations} that  for any $u_i\in V^{\circ}$, we have  $\ord_{T_i}(f_K)=\ord_{S_i}(f)$. 
\end{proof}

\bigskip

\begin{proof}[Proof of Theorem \ref{t-JMconjecture}] Let $w$ compute the log canonical threshold of $\fa_{\bullet}$ on $(X,\Delta)$ (see \cite[Theorem 7.3]{JM-qmvaluations}) with $c_{X}(w)=\eta$.  We can replace $\fa_{\bullet}$ by $\fa^{w}_{\bullet}$ (see Lemma \ref{l-enlarge}) and localize at $\eta$, thus we reduce to  the case that $\fa_{\bullet}$ is a graded sequence of $\fm_{\eta}$-primary ideals where $\fm_{\eta}$ is the maximal ideal on a local ring of an essentially finite type.

Then we can apply Proposition \ref{p-approxlc} which says that there exists a valuation $v$ that can be written as the limit of a sequence of valuations with the form $c_j\cdot \ord_{S_j}$ for  Koll\'ar components $\{S_j\}$ such that $A_{X,\Delta}(v)=A_{X,\Delta}(w)$, $v\ge w$, and $v$ also calculates the log canonical thresholds of $\fa_{\bullet}$. By Lemma \ref{l-order} , there exists $\delta>0$ such that $c_j\cdot \ord_{S_j}(\fm_x)>\delta$. Thus by Theorem \ref{t-quasi-monomial},  $v$ has to be quasi-monomial. 
\end{proof}
\begin{proof}[Proof of Theorem \ref{t-minimizing}] By Theorem \ref{t-JMconjecture}, we know for a minimizer $w$, there exists a quasi-monomial valuation $v$ which computes the log canonical threshold of $\fa^w_{\bullet}$. Since $w$ is a minimizer of $\hvol_{(X,\Delta),x}$, we conclude $v=\lambda w$ by \cite[Lemma 4.7]{Blu-Existence} for some $\lambda>0$.
\end{proof}

\begin{rem}\label{r-existence}
In \cite{Blu-Existence}, to show the existence of the minimizer, there is a technical assumption that the ground field $k$ has to be uncountable. Our approach can indeed remove this assumption. 

More precisely, we can always find a sequence of Koll\'ar component $S_i$ such that $\lim_i \hvol(S_i)=\inf \hvol(v)$ (see \cite[Lem. 3.8]{LX-SDCI}). 
Therefore, after passing to a further infinite subsequence, as in the proof of Theorem  \ref{t-quasi-monomial}, we can assume there is a family $(X\times V, \Delta\times V+\frac{1}{N}D)\to V$ which admits a fiberwise log resolution, and $S_i$ is an lc place of $(X,\Delta+\frac{1}{N}D_{u_i})$ for some $u_i\in V$. Fix a closed point $u\in V$, then as before, $S_i$ yields a divisor $T_i$ which is a lc place of $(X,\Delta+\frac{1}{N}D_u)$ and yields the same point as $S_i$ under the correspondence \eqref{e-dualcomplexes}. By Theorem \ref{t-invlocal}, $\hvol(\ord_{S_i})=\hvol(\ord_{T_i})$. Thus $w_i:=\frac{1}{A_{X,\Delta}(T_i)}\ord_{T_i}$ has a limit $w$.
Since $\vol(\cdot )$ is continuous on a dual complex by \cite[Corollary D]{BFJ-Izumi}, which implies that $\hvol(\cdot)$ is also continous, we have
$$\hvol(w)=\lim_i\hvol(\ord_{T_i})=\lim_i\hvol(\ord_{S_i})=\inf \hvol(v),$$
i.e. $w$ is a minimizer of $\hvol_{(X,\Delta),x}$.
\end{rem}

\section{Family version}
In this section, we will use the techniques  developed  in the previous section to study a $\mathbb{Q}$-Gorenstein family of klt singularities, and prove the normalized volume function is a constructible function. As a consequence, we obtain the openness of the K-semistable locus among a $\mathbb{Q}$-Gorenstein family of log Fano pairs. 

Let $(X,\Delta)\to B$ be family of klt singularities over a smooth base $B$ with a section $\sigma \colon B\to X$. 
\begin{lem}\label{l-proper}
 There is a uniform positive constant $\delta>0$ depending only on $B\subset (X,\Delta)$, such that for any geometric point $s\to B$, and a valuation $v_s\in \Val_{X_s, x_s}$ with $A_{X_s,\Delta_s}(v_s)<+\infty$, then 
$$\hvol(v_s)\cdot v_s(\fm_s)>\delta\cdot A_{X_s,\Delta_s}(v_s).$$   
\end{lem}
\begin{proof}This is a family version of the proper estimate in \cite[Theorem 1.1]{Li-minimizer}. See \cite[Theorem 21]{BL-volumelower} for a proof.
\end{proof}

The following statement is a generalization of Proposition \ref{p-complement}.

\begin{prop}\label{p-familykollar}
Let $B\subset (X,\Delta)\to B$ be a $\mathbb{Q}$-Gorenstein family of klt singularities over a smooth base $B$. Fix a positive number $C$. There is a family of Cartier divisors $D\subset X\times_BV$ over a finite type variety $\pi\colon V\to B$ and a positive number $N$, such that for a geometric point $s\to B$ and a Koll\'ar components $S_{s}$ over $(X_s,\Delta_s)$ with $\hvol(\ord_{S_s})\le C$, then there exists a point $u\in V\times_B\{s\}$ such that if we base change $(X_s,\Delta_s)$ and $S_s$ to $u$, $S_u$ is a log canonical place of the log canonical pair $(X_u,\Delta_u+\frac{1}{N}D_u)$, where $D_u:=D\times_V\{u\}$. 
\end{prop}
\begin{proof} Let $\delta$ be the constant as in Lemma \ref{l-proper} and $\delta_0=\frac{\delta}{C}>0$. Then it follows from our assumption $\hvol(\ord_{S_s})\le C$ that
$$\frac{1}{A_{X_s,\Delta_s}(S_s)}\ord_{S_s}(\fm_s)\ge \delta_0.$$
Fix $M$ such that $M\delta_0>N:=r N_0$, where $N_0$ is the constant  given by Lemma \ref{t-localcomplement} which only depends on the dimension of $X_s$ and the coefficients of $\Delta$ and $r$ is a positive integer such that $r(K_X+\Delta)$ is Cartier.
By shrinking $B$, we can assume $B={\rm Spec}(T)$, $\mathcal{O}_{X}/(\fm_{\sigma(B)})^M$ is a free $T$-module with a basis $[g_1],...,[g_m]$ for $g_i\in \mathcal{O}_{B,X}$.

Let ${E}\to U:= (\mathbb{A}^m\setminus\{0\})_T$ be the space, such that over point $t=(t_1,...,t_m)\in U$, the fiber $E_t$ parametrises the divisor of $(\sum^m_{j=1}t_jg_j=0)$. 
For the family 
$$(X\times_B  U, \Delta\times_B  U+ {E} ) \xrightarrow[]{} U\xrightarrow[]{\pi} B,$$
by Lemma \ref{l-lctfamily}, there is a constructible set $V\subset U$ such that for a point $u\to U$ factors through $V$ if and only if
$$\lct(X_u, \Delta_u; {E}_u)=\frac{1}{N}. 
$$

Let $D:=E\times_U V$. If  a Koll\'ar component $S_{s}$ over $(X_s,\Delta_s)$ for a geometric point $s\to B$ satisfies
$$\hvol_{X_s,\Delta_s}(\ord_{S_s})<C,$$ 
then the same argument for Proposition \ref{l-order} shows that $S_s$ is a log canonical place of the pair $(X_s,\Delta_s+\frac{1}{N}D')$ for some $D'={\rm div}(g)$ with $g=\sum^m_{i=1} \lambda_i(g_i)_s$ for some $\lambda_i\in k(s)$, where $(g_i)_s$ is the reduction of $g_i$ under the morphism  $\mathcal{O}_{B,X} \to \cO_{s, X_s}$.  Thus $D'=D_u$ for $u=(\lambda_1,...,\lambda_m)$ over $s\to B$.
\end{proof}
\begin{proof}[Proof of Theorem \ref{t-volumeconstruct}]
Let $C=n^n+1$, then we know for any geometric point $s\in B$, $\hvol(s,X_s,\Delta_s)< C$ by \cite[Theorem 1.6]{LX-cubic}. Apply Proposition \ref{p-familykollar} to such $C$ and let $(D\subset X\times _BV)\to V$ be the family of divisors given by it. Then after stratifying the base $V$ into a disjoint union of  finitely many constructible subsets and taking finite \'etale coverings, we can assume there exists a decomposition $V=\bigsqcup_{\alpha} V_{\alpha}$ into irreducible smooth strata $V_{\alpha}$ such that for each $\alpha$, $\big(X\times_B V_{\alpha},{\rm Supp} (\Delta\times_B V_{\alpha}+\frac{1}{N}D)\big)$ admits a fiberwise log resolution $\mu \colon Y_{\alpha} \to X\times_B V_{\alpha}$ over $V_{\alpha}$ with a simple normal crossing exceptional divisor $E_{\alpha}=\sum^k_{i=1}E_i$ (see Definition-Lemma \ref{l-logresol}). We choose $\Gamma_{\alpha}\subset E_{\alpha}$ to be the subdivisor given by the components  with log discrepancy 0 with respect to $(X\times_B V, \Delta\times_B V+\frac{1}{N}D)$. Moreover,   by the Noether induction, we can shrink $B$ and assume each $V_\alpha\to B$ is surjective. 

 Then for any geometric point $s\to B$ and a Koll\'ar component  $S_i$ over $(X_s, \Delta_s)$ with $\hvol(\ord_{S_i})\le C$, Proposition \ref{p-familykollar} implies that there is a point $u_i\in V_{\alpha}\times_B \{s\}$ such that $(X_{u_i}, \Delta_{u_i}+\frac{1}{N}D_{u_i})$ is log canonical and $S_{u_i}$ is a log canonical place of the pair, where $(X_{u_i}, \Delta_{u_i})$ and $S_{u_i}$ are the base changes of $(X_{s},\Delta_s)$ and $S_i$ over $u_i$.
  Define $(Y_{u_i}, E_{u_i}):=(Y_{\alpha},E_{\alpha})\times_{V_{\alpha}}\{u_i\} $. Since $\mu_{u_i}\colon Y_{u_i}\to (X_{{u_i}}, \Delta_{u_i}+\frac{1}{N}D_{u_i})$ is a log resolution, we know that $S_i$ will be a toroidal divisor over $(Y_{u_i}, E_{u_i})$, which then yields a toroidal divisor $T_i$ over $(Y_{\alpha},E_{\alpha})$ whose corresponding valuation is contained in  $i_{X,\Delta}(\cD(\Gamma_{\alpha}))\times {\mathbb{R}_{>0}},$ such that $S_i$ is given by the restriction of $T_i$ over $u_i$.

For any  toroidal divisor $T$  with  $\ord_{T}\in i_{X,\Delta}(\cD(\Gamma_{\alpha}))\times {\mathbb{R}_{>0}},$
 since 
 $$\big(X\times_B V_{\alpha}, \Delta\times_B V_{\alpha}+(\frac{1}{N}-\epsilon)D\big)$$ is klt for any  $\frac{1}{N}\ge\epsilon>0$, and we can choose $\epsilon$ sufficiently  small such that 
$$A_{X\times_BV_{\alpha}, \Delta\times_BV_{\alpha}+(\frac{1}{N}-\epsilon)D}(T)<1,$$ by Theorem \ref{t-invlocal} we conclude that for  any $u\in V_{\alpha}$, the function
$$u\in V_{\alpha}\to \hvol_{X_u,\Delta_u}(T_u)$$ is a constant function on $V_{\alpha}$. 
Thus,  for each fixed $\alpha$, the function
$$v_{\alpha}\colon u\in V_{\alpha}\to\inf_{T_u}\ \{\hvol_{X_u,\Delta_u}(\ord_{T_u})\  |\  \ord_{T_u}\in i_{X_u,\Delta_u}(\cD(\Gamma|_{Y_u}))\times \mathbb{R}_{>0} \}$$
is a constant function. 

For any geometric point $s\colon \Spec k\to B$, since the point can be lifted to $\Spec k\to V_{\alpha}$ for any $\alpha$, then
\begin{eqnarray*}
\hvol(s,X_s,\Delta_s)&=&\inf_{S_i}\{\ \hvol_{X_s,\Delta_s}(\ord_{S_i})\  |\mbox{ Koll\'ar components $S_i$ with } \hvol(S_i)\le n^n+1 \}\\
 &\ge&\min_{\alpha} \{v_{\alpha}\}\\
 &\ge & \hvol(s,X_s,\Delta_s),
\end{eqnarray*}
where the second relation holds since
any Koll\'ar component over $(X_s,\Delta_s)$ with normalised volume at most $n^n+1$ is isomorphic to $T_u$ for some $u\in V_{\alpha}$ mapping to $s$; and the last relation follows from  that there is a lift $\Spec k \to V_{\alpha}$ of $s\to B$ for every $\alpha$. 
Thus $s\to \hvol({s},X_{{s}},\Delta_{{s}})$ is a constant on all geometric points after we shrink $B$ to a nonempty open set, which implies it is  constructible. 
\end{proof}
It is known that $\hvol(s,X_s,\Delta_s)$ is also lower-semicontinous by \cite[Theorem 1]{BL-volumelower}. Indeed, with Theorem \ref{t-volumeconstruct}, to see this we only need the weaker result that the normalized volume does not increase under a specialization of singularities.

\medskip

It is well known that Theorem \ref{t-openness} follows from Theorem \ref{t-volumeconstruct} via the cone construction. 
\begin{proof}[Proof of Theorem \ref{t-openness}]
We  take the relative cone over $ {Y}:=C(X/B, -r(K_{X}+\Delta))$ for sufficiently divisible $r$, i.e.
$Y={\rm Spec}_{\cO_{B}}\bigoplus^{\infty}_{m=0}\big(f_*(-rm(K_{X}+\Delta) )\big),$ 
and we can pull back the base $\Delta$ to get a boundary $\Delta_{Y}$. It has a section $B\to (Y,\Delta_{Y})$ given by the  cone vertices and makes it a $\mathbb{Q}$-Gorenstein family of klt singularities. 

For any $s\in B$, 
$$\hvol (s, Y_s,\Delta_{Y_s})\le \hvol_{(Y_s,\Delta_{Y_s}),s}(\ord_{V_s})=\frac{1}{r}(-K_{X_s}-\Delta_s)^n,$$
where $V_s$ is the divisor obtained by blowing up the vertex. By Proposition \ref{p-cone}, for any geometric point $s\in B$, $(X_s,\Delta_s)$ is K-semistable if and only if the equality holds. 

 By Theorem \ref{t-volumeconstruct} and \cite[Theorem 1]{BL-volumelower}, $\hvol(s, Y_s,\Delta_{Y_s})$ is constructible and lower semi-continuous. Therefore, there is an open set $B^{\circ}$ of $B$, such that $\hvol(s, Y_s,\Delta_{Y_s})$ takes the possibly maximal value $\frac{1}{r}(-K_{X_s}-\Delta_s)^n$, which is precisely the locus where the geometric fibers are K-semistable. 
\end{proof}
\begin{rem}\label{r-nonreduced}
When $\Delta=0$, the proof of Theorem \ref{t-openness} clearly can also be applied to any locally stable family of klt Fano varieties (see Definition \ref{d-kollar}).
\end{rem}

It is known that openness of K-semistability was the last missing ingredient to prove Theorem \ref{t-goodmoduli} (see \cite[Corollary 1.2]{ABHX-reductivity}). An outline of the construction is given in \cite{BX-uniqueness} for locally stable families of uniformly K-stable Fano varieties. To get  Theorem \ref{t-goodmoduli}, we only need to replace the ingredients, and then the same argument applies.  
\begin{proof}[Proof of Theorem \ref{t-goodmoduli}]
We follow the proof of \cite[Corollary 1.4]{BX-uniqueness}. Replacing the uniform K-stability by the K-semistability, and \cite{BL-openness} by Theorem \ref{t-openness} (see Remark \ref{r-nonreduced}), we conclude that $\mathfrak{X}^{\rm kss}_{n,V}$ is parametrised by an Artin stack of finite type over $k$. Then by \cite[Theorem 1.1]{BX-uniqueness} and \cite[Corollary 1.2]{ABHX-reductivity}, we know the good moduli space ${X}^{\rm kps}_{n,V}$ exists and is separated. 
\end{proof}


\bibliography{refxu} 

\begin{bibdiv}
\begin{biblist}

\bib{ABHX-reductivity}{article}{
      author={Alper, Jarod},
      author={Blum, Harold},
      author={Halpern-Leistner, Daniel},
      author={Xu, Chenyang},
       title={Reductivity of the automorphism group of {K}-polystable {F}ano
  varieties},
        date={2019},
}

\bib{BCHM}{article}{
      author={Birkar, Caucher},
      author={Cascini, Paolo},
      author={Hacon, Christopher~D.},
      author={McKernan, James},
       title={Existence of minimal models for varieties of log general type},
        date={2010},
        ISSN={0894-0347},
     journal={J. Amer. Math. Soc.},
      volume={23},
      number={2},
       pages={405\ndash 468},
         url={https://doi.org/10.1090/S0894-0347-09-00649-3},
      review={\MR{2601039}},
}

\bib{BdFFU-valuation}{incollection}{
      author={Boucksom, S.},
      author={de~Fernex, T.},
      author={Favre, C.},
      author={Urbinati, S.},
       title={Valuation spaces and multiplier ideals on singular varieties},
        date={2015},
   booktitle={Recent advances in algebraic geometry},
      series={London Math. Soc. Lecture Note Ser.},
      volume={417},
   publisher={Cambridge Univ. Press, Cambridge},
       pages={29\ndash 51},
      review={\MR{3380442}},
}

\bib{Berndtsson-openness}{article}{
      author={Berndtsson, Bo},
       title={The openness conjecture and complex {B}runn-{M}inkowski
  inequalities},
        date={2015},
      volume={10},
       pages={29\ndash 44},
      review={\MR{3587460}},
}

\bib{BFJ-Izumi}{incollection}{
      author={Boucksom, S\'{e}bastien},
      author={Favre, Charles},
      author={Jonsson, Mattias},
       title={A refinement of {I}zumi's theorem},
        date={2014},
   booktitle={Valuation theory in interaction},
      series={EMS Ser. Congr. Rep.},
   publisher={Eur. Math. Soc., Z\"{u}rich},
       pages={55\ndash 81},
      review={\MR{3329027}},
}

\bib{Bir-BABI}{article}{
      author={Birkar, Caucher},
       title={Anti-pluricanonical systems on {F}ano varieties},
        date={2019},
        ISSN={0003-486X},
     journal={Ann. of Math. (2)},
      volume={190},
      number={2},
       pages={345\ndash 463},
         url={https://doi.org/10.4007/annals.2019.190.2.1},
      review={\MR{3997127}},
}

\bib{BL-volumelower}{article}{
      author={Blum, Harold},
      author={Liu, Yuchen},
       title={The normalized volume of a singularity is lower semicontinuous},
        date={2018},
     journal={to appear in {J. Euro. Math. Soc.}},
  note={\href{https://arxiv.org/abs/1802.09658}{\textsf{arXiv:1802.09658}}},
}

\bib{BL-openness}{article}{
      author={Blum, Harold},
      author={Liu, Yuchen},
       title={Openness of uniform {K}-stability in families of
  $\mathbb{Q}$-{F}ano varieties},
        date={2018},
  note={\href{https://arxiv.org/abs/1808.09070}{\textsf{arXiv:1808.09070}}},
}

\bib{Blu-Existence}{article}{
      author={Blum, Harold},
       title={Existence of valuations with smallest normalized volume},
        date={2018},
        ISSN={0010-437X},
     journal={Compos. Math.},
      volume={154},
      number={4},
       pages={820\ndash 849},
         url={https://doi.org/10.1112/S0010437X17008016},
      review={\MR{3778195}},
}

\bib{BLX-openness}{article}{
      author={Blum, Harold},
      author={Liu, Yuchen},
      author={Xu, Chenyang},
       title={Openness of {K}-semistability for {F}ano varieties},
        date={2019},
  note={\href{https://arxiv.org/abs/1907.02408}{\textsf{arXiv:1907.02408}}},
}

\bib{BX-uniqueness}{article}{
      author={Blum, Harold},
      author={Xu, Chenyang},
       title={Uniqueness of {K}-polystable degenerations of fano varieties},
        date={2019},
        ISSN={0003-486X},
     journal={Ann. of Math. (2)},
      volume={190},
      number={2},
       pages={609\ndash 656},
         url={https://doi.org/10.4007/annals.2019.190.2.4},
      review={\MR{3997130}},
}

\bib{Cutkosky-multiplicities}{article}{
      author={Cutkosky, Steven~Dale},
       title={Multiplicities associated to graded families of ideals},
        date={2013},
        ISSN={1937-0652},
     journal={Algebra Number Theory},
      volume={7},
      number={9},
       pages={2059\ndash 2083},
         url={https://doi.org/10.2140/ant.2013.7.2059},
      review={\MR{3152008}},
}

\bib{Cutkosky-twodimensional}{incollection}{
      author={Cutkosky, Steven~Dale},
       title={On finite and nonfinite generation of associated graded rings of
  {A}bhyankar valuations},
        date={2018},
   booktitle={Singularities, algebraic geometry, commutative algebra, and
  related topics},
   publisher={Springer, Cham},
       pages={481\ndash 490},
      review={\MR{3839808}},
}

\bib{dFEM-ACC}{article}{
      author={de~Fernex, Tommaso},
      author={Ein, Lawrence},
      author={Musta\c{t}\u{a}, Mircea},
       title={Shokurov's {ACC} conjecture for log canonical thresholds on
  smooth varieties},
        date={2010},
        ISSN={0012-7094},
     journal={Duke Math. J.},
      volume={152},
      number={1},
       pages={93\ndash 114},
         url={https://doi.org/10.1215/00127094-2010-008},
      review={\MR{2643057}},
}

\bib{dFKX-dualcomplex}{incollection}{
      author={de~Fernex, Tommaso},
      author={Koll\'{a}r, J\'{a}nos},
      author={Xu, Chenyang},
       title={The dual complex of singularities},
        date={2017},
   booktitle={Higher dimensional algebraic geometry---in honour of {P}rofessor
  {Y}ujiro {K}awamata's sixtieth birthday},
      series={Adv. Stud. Pure Math.},
      volume={74},
   publisher={Math. Soc. Japan, Tokyo},
       pages={103\ndash 129},
      review={\MR{3791210}},
}

\bib{DK-lct}{article}{
      author={Demailly, Jean-Pierre},
      author={Koll\'{a}r, J\'{a}nos},
       title={Semi-continuity of complex singularity exponents and
  {K}\"{a}hler-{E}instein metrics on {F}ano orbifolds},
        date={2001},
        ISSN={0012-9593},
     journal={Ann. Sci. \'{E}cole Norm. Sup. (4)},
      volume={34},
      number={4},
       pages={525\ndash 556},
         url={https://doi.org/10.1016/S0012-9593(01)01069-2},
      review={\MR{1852009}},
}

\bib{ELS-localvolume}{article}{
      author={Ein, Lawrence},
      author={Lazarsfeld, Robert},
      author={Smith, Karen~E.},
       title={Uniform approximation of {A}bhyankar valuation ideals in smooth
  function fields},
        date={2003},
        ISSN={0002-9327},
     journal={Amer. J. Math.},
      volume={125},
      number={2},
       pages={409\ndash 440},
  url={http://muse.jhu.edu/journals/american_journal_of_mathematics/v125/125.2ein.pdf},
      review={\MR{1963690}},
}

\bib{GZ-openness}{article}{
      author={Guan, Qi'an},
      author={Zhou, Xiangyu},
       title={A proof of {D}emailly's strong openness conjecture},
        date={2015},
        ISSN={0003-486X},
     journal={Ann. of Math. (2)},
      volume={182},
      number={2},
       pages={605\ndash 616},
         url={https://doi.org/10.4007/annals.2015.182.2.5},
      review={\MR{3418526}},
}

\bib{Hiep-openness}{article}{
      author={Hiep, Pham~Hoang},
       title={The weighted log canonical threshold},
        date={2014},
        ISSN={1631-073X},
     journal={C. R. Math. Acad. Sci. Paris},
      volume={352},
      number={4},
       pages={283\ndash 288},
         url={https://doi.org/10.1016/j.crma.2014.02.010},
      review={\MR{3186914}},
}

\bib{HMX-auto}{article}{
      author={Hacon, Christopher~D.},
      author={McKernan, James},
      author={Xu, Chenyang},
       title={On the birational automorphisms of varieties of general type},
        date={2013},
        ISSN={0003-486X},
     journal={Ann. of Math. (2)},
      volume={177},
      number={3},
       pages={1077\ndash 1111},
         url={https://doi.org/10.4007/annals.2013.177.3.6},
      review={\MR{3034294}},
}

\bib{Jiang-bounded}{article}{
      author={Jiang, Chen},
       title={Boundedness of $\mathbb{Q}$-{F}ano varieties with degrees and
  alpha-invariants bounded from below},
        date={2017},
     journal={to appear in Ann. Sci. \'{E}c. Norm. Sup\'{e}r. (4)},
  note={\href{https://arxiv.org/abs/1705.02740}{\textsf{arXiv:1705.02740}}},
}

\bib{JM-qmvaluations}{article}{
      author={Jonsson, Mattias},
      author={Musta\c{t}\u{a}, Mircea},
       title={Valuations and asymptotic invariants for sequences of ideals},
        date={2012},
        ISSN={0373-0956},
     journal={Ann. Inst. Fourier (Grenoble)},
      volume={62},
      number={6},
       pages={2145\ndash 2209 (2013)},
         url={https://doi.org/10.5802/aif.2746},
      review={\MR{3060755}},
}

\bib{JM-openness}{article}{
      author={Jonsson, Mattias},
      author={Musta\c{t}\u{a}, Mircea},
       title={An algebraic approach to the openness conjecture of {D}emailly
  and {K}oll\'{a}r},
        date={2014},
        ISSN={1474-7480},
     journal={J. Inst. Math. Jussieu},
      volume={13},
      number={1},
       pages={119\ndash 144},
         url={https://doi.org/10.1017/S1474748013000091},
      review={\MR{3134017}},
}

\bib{Kawakita-inversion}{article}{
      author={Kawakita, Masayuki},
       title={Inversion of adjunction on log canonicity},
        date={2007},
        ISSN={0020-9910},
     journal={Invent. Math.},
      volume={167},
      number={1},
       pages={129\ndash 133},
         url={https://doi.org/10.1007/s00222-006-0008-z},
      review={\MR{2264806}},
}

\bib{KollarMori98}{book}{
      author={Koll\'{a}r, J\'{a}nos},
      author={Mori, Shigefumi},
       title={Birational geometry of algebraic varieties},
      series={Cambridge Tracts in Mathematics},
   publisher={Cambridge University Press, Cambridge},
        date={1998},
      volume={134},
        ISBN={0-521-63277-3},
         url={https://doi.org/10.1017/CBO9780511662560},
        note={With the collaboration of C. H. Clemens and A. Corti, Translated
  from the 1998 Japanese original},
      review={\MR{1658959}},
}

\bib{Kollar-genericlimit}{article}{
      author={Koll\'{a}r, J\'{a}nos},
       title={Which powers of holomorphic functions are integrable?},
        date={2008},
        note={\href{
  https://arxiv.org/abs/0805.0756}{\textsf{arXiv:0805.0756}}},
}

\bib{Kollar13}{book}{
      author={Koll\'{a}r, J\'{a}nos},
       title={Singularities of the minimal model program},
      series={Cambridge Tracts in Mathematics},
   publisher={Cambridge University Press, Cambridge},
        date={2013},
      volume={200},
        ISBN={978-1-107-03534-8},
         url={https://doi.org/10.1017/CBO9781139547895},
        note={With a collaboration of S\'{a}ndor Kov\'{a}cs},
      review={\MR{3057950}},
}

\bib{Kollar-modulibook}{book}{
      author={Koll\'{a}r, J\'{a}nos},
       title={Families of varieties of general type},
        date={2020},
  note={\href{https://web.math.princeton.edu/~kollar/book/modbook20170720-hyper.pdf}{\textsf{modulibook}}},
}

\bib{Laz-positivityII}{book}{
      author={Lazarsfeld, Robert},
       title={Positivity in algebraic geometry. {II}},
      series={Ergebnisse der Mathematik und ihrer Grenzgebiete. 3. Folge. A
  Series of Modern Surveys in Mathematics [Results in Mathematics and Related
  Areas. 3rd Series. A Series of Modern Surveys in Mathematics]},
   publisher={Springer-Verlag, Berlin},
        date={2004},
      volume={49},
        ISBN={3-540-22534-X},
         url={https://doi.org/10.1007/978-3-642-18808-4},
        note={Positivity for vector bundles, and multiplier ideals},
      review={\MR{2095472}},
}

\bib{Lempert-openness}{incollection}{
      author={Lempert, L\'{a}szl\'{o}},
       title={Modules of square integrable holomorphic germs},
        date={2017},
   booktitle={Analysis meets geometry},
      series={Trends Math.},
   publisher={Birkh\"{a}user/Springer, Cham},
       pages={311\ndash 333},
      review={\MR{3773623}},
}

\bib{Li-K=M}{article}{
      author={Li, Chi},
       title={K-semistability is equivariant volume minimization},
        date={2017},
        ISSN={0012-7094},
     journal={Duke Math. J.},
      volume={166},
      number={16},
       pages={3147\ndash 3218},
         url={https://doi.org/10.1215/00127094-2017-0026},
      review={\MR{3715806}},
}

\bib{Li-minimizer}{article}{
      author={Li, Chi},
       title={Minimizing normalized volumes of valuations},
        date={2018},
        ISSN={0025-5874},
     journal={Math. Z.},
      volume={289},
      number={1-2},
       pages={491\ndash 513},
         url={https://doi.org/10.1007/s00209-017-1963-3},
      review={\MR{3803800}},
}

\bib{LL-K=M}{article}{
      author={Li, Chi},
      author={Liu, Yuchen},
       title={K\"{a}hler-{E}instein metrics and volume minimization},
        date={2019},
        ISSN={0001-8708},
     journal={Adv. Math.},
      volume={341},
       pages={440\ndash 492},
         url={https://doi.org/10.1016/j.aim.2018.10.038},
      review={\MR{3872852}},
}

\bib{LLX-Tiansurvey}{incollection}{
      author={Li, Chi},
      author={Liu, Yuchen},
      author={Xu, Chenyang},
       title={A guided tour to normalized volume},
        date={2017},
  note={\href{https://arxiv.org/abs/1806.07112}{\textsf{arXiv:1806.07112}}},
}

\bib{LM-okounkovbody}{article}{
      author={Lazarsfeld, Robert},
      author={Musta\c{t}\u{a}, Mircea},
       title={Convex bodies associated to linear series},
        date={2009},
        ISSN={0012-9593},
     journal={Ann. Sci. \'{E}c. Norm. Sup\'{e}r. (4)},
      volume={42},
      number={5},
       pages={783\ndash 835},
         url={https://doi.org/10.24033/asens.2109},
      review={\MR{2571958}},
}

\bib{LX-specialTC}{article}{
      author={Li, Chi},
      author={Xu, Chenyang},
       title={Special test configuration and {K}-stability of {F}ano
  varieties},
        date={2014},
        ISSN={0003-486X},
     journal={Ann. of Math. (2)},
      volume={180},
      number={1},
       pages={197\ndash 232},
         url={https://doi.org/10.4007/annals.2014.180.1.4},
      review={\MR{3194814}},
}

\bib{LX-SDCI}{article}{
      author={Li, Chi},
      author={Xu, Chenyang},
       title={Stability of valuations and {K}oll\'ar components},
        date={2016},
     journal={to appear in {J. Euro. Math. Soc.}},
  note={\href{https://arxiv.org/abs/1604.05398}{\textsf{arXiv:1604.05398}}},
}

\bib{LX-SDCII}{article}{
      author={Li, Chi},
      author={Xu, Chenyang},
       title={Stability of valuations: Higher rational rank},
        date={2018},
     journal={Peking Math. J.},
      volume={1},
      number={1},
       pages={1\ndash 79},
}

\bib{LX-cubic}{article}{
      author={Liu, Yuchen},
      author={Xu, Chenyang},
       title={K-stability of cubic threefolds},
        date={2019},
        ISSN={0012-7094},
     journal={Duke Math. J.},
      volume={168},
      number={11},
       pages={2029\ndash 2073},
         url={https://doi.org/10.1215/00127094-2019-0006},
      review={\MR{3992032}},
}

\bib{Naka-zariski}{book}{
      author={Nakayama, Noboru},
       title={Zariski-decomposition and abundance},
      series={MSJ Memoirs},
   publisher={Mathematical Society of Japan, Tokyo},
        date={2004},
      volume={14},
        ISBN={4-931469-31-0},
      review={\MR{2104208}},
}

\bib{Prok-plt}{incollection}{
      author={Prokhorov, Yu.~G.},
       title={Blow-ups of canonical singularities},
        date={2000},
   booktitle={Algebra ({M}oscow, 1998)},
   publisher={de Gruyter, Berlin},
       pages={301\ndash 317},
      review={\MR{1754677}},
}

\bib{Shokurov-threedim}{article}{
      author={Shokurov, V.~V.},
       title={Three-dimensional log perestroikas},
        date={1992},
        ISSN={1607-0046},
     journal={Izv. Ross. Akad. Nauk Ser. Mat.},
      volume={56},
      number={1},
       pages={105\ndash 203},
         url={https://doi.org/10.1070/IM1993v040n01ABEH001862},
      review={\MR{1162635}},
}

\bib{Xu-kollar}{article}{
      author={Xu, Chenyang},
       title={Finiteness of algebraic fundamental groups},
        date={2014},
        ISSN={0010-437X},
     journal={Compos. Math.},
      volume={150},
      number={3},
       pages={409\ndash 414},
         url={https://doi.org/10.1112/S0010437X13007562},
      review={\MR{3187625}},
}

\end{biblist}
\end{bibdiv}
\bibliographystyle{ieeetr}

\end{document}